\documentclass[11pt]{article}

\usepackage{fullpage}
\usepackage{amsfonts,amssymb,stmaryrd,amsthm,amsmath}
\usepackage{frcursive}

\usepackage{graphicx,tikz,tikz-cd, tabularx}
\usetikzlibrary{matrix,arrows,decorations.markings}

\newcommand\longto{{\longrightarrow}}

\newcommand\RR{\mathbb{R}}\newcommand\PP{\mathbb P}
\newcommand\CC{\mathbb{C}}
\newcommand\NN{\mathbb{N}}\newcommand\ZZ{\mathbb{Z}}

\newcommand\inv{{^{-1}}}

\newcommand{\ts}{\tilde \sigma}
\newcommand\hy{{\widehat y}}\newcommand\halpha{{\widehat \alpha}}
\newcommand\hbeta{{\widehat \beta}}
\newcommand\hs{{\widehat s}}
\newcommand\hPhi{{\widehat\Phi}}
\newcommand\hw{{\widehat w}}\newcommand\hW{{\widehat W}}
\newcommand\hnu{{\widehat\nu}}
\newcommand\hGB{{\widehat G/\widehat B}}
\newcommand\hV{{\widehat V}}
\newcommand\hG{{\widehat G}} \newcommand\hB{{\widehat B}}\newcommand\hT{{\widehat
    T}} 

\newcommand\Lie{{\operatorname{Lie}}}
\newcommand\Mult{{\operatorname{Mult}}}
\newcommand\SL{\operatorname{SL}}\newcommand\Sp{\operatorname{Sp}}\newcommand\PSL{\operatorname{PSL}}
\newcommand\SO{\operatorname{SO}}\newcommand\Spin{\operatorname{Spin}}

\newcommand\Li{{\mathcal
    L}}\newcommand\Sc{{\mathcal S}}\newcommand\Dc{{\mathcal D}}
\newcommand\Div{{\mathcal D}}
\newcommand\Orb{{\mathcal O}}
\newcommand\Hom{{\operatorname{Hom}}}
\newcommand\GL{{\operatorname{GL}}}

\renewcommand\sl{{\mathfrak{sl}}}

\newcommand{\ema}{\epsilon_{-\ha}}
\newcommand{\eb}{{\epsilon_{\beta}}}
\newcommand{\emb}{\epsilon_{- \beta}}

\newcommand{\ea}{\epsilon_{\ha}}

\renewcommand{\lg}{{\mathfrak g}}\newcommand{\hlg}{{\widehat{\mathfrak g}}}
\newcommand{\lh}{{\mathfrak h}}

\newcommand{\Pha}{P_{\widehat \alpha}}

\newcommand{\ha}{{\widehat \alpha}}
\newcommand{\hb}{\widehat \beta}

\newcommand{\Hc}{\mathcal H}

\usepackage[utf8]{inputenc}  

\usepackage[T1]{fontenc}

\usepackage{hyperref}
\hypersetup{pdfauthor={Name}}

\newtheorem{prop}{Proposition}
\newtheorem{theo}[prop]{Theorem}
\newtheorem{lemma}[prop]{Lemma}

\theoremstyle{remark}
\newtheorem*{remark}{Remark}

\newenvironment{defi}{\noindent{\bf Definition.}}{~\\}
\newcommand\DynkinNodeSize{2mm}
\newcommand\DynkinArrowLength{3mm}
\tikzset{
  dnode/.style={
    circle,
    inner sep=0pt,
    minimum size=\DynkinNodeSize,
    fill=white,
    draw},
  middlearrow/.style={
    decoration={markings,
      mark=at position 0.6 with
      {\draw (0:0mm) -- +(+135:\DynkinArrowLength); \draw (0:0mm) -- +(-135:\DynkinArrowLength);},
    },
    postaction={decorate}
  },
  leftrightarrow/.style={
    decoration={markings,
      mark=at position 0.999 with
      {
      \draw (0:0mm) -- +(+135:\DynkinArrowLength); \draw (0:0mm) -- +(-135:\DynkinArrowLength);
      },
      mark=at position 0.001 with
      {
      \draw (0:0mm) -- +(+45:\DynkinArrowLength); \draw (0:0mm) -- +(-45:\DynkinArrowLength);
      },
    },
    postaction={decorate}
  },
  sedge/.style={
  },
  dedge/.style={
    middlearrow,
    double distance=0.5mm,
  },
  tedge/.style={
    middlearrow,
    double distance=1.0mm+\pgflinewidth,
    postaction={draw}, 
  },
  infedge/.style={
    leftrightarrow,
    double distance=0.5mm,
  }
}

\newcommand\revddots{\mathinner{\mkern1mu\raise\p@\vbox{\kern7\p@\hbox{.}}\mkern2mu\raise4\p@\hbox{.}\mkern2mu\raise7\p@\hbox{.}\mkern1mu}}
\makeatletter
\def\revddots{\mathinner{\mkern1mu\raise\p@\vbox{\kern7\p@\hbox{.}}\mkern2mu\raise4\p@\hbox{.}\mkern2mu\raise7\p@\hbox{.}\mkern1mu}}
\makeatother

\newcommand\base{{\mathcal B}}
\newcommand\diag{{\operatorname{diag}}}
\newcommand{\scal}[1]{\langle #1 \rangle}

\begin{document}
\title{On the multiplicity spaces for branching to a spherical
  subgroup of minimal rank}
\author{Luca Francone and Nicolas Ressayre}

\maketitle

\begin{abstract}
Let $\hlg$ be a complex semi-simple 
Lie algebra and $\lg$ be a
semisimple subalgebra of $\hlg$. Consider the branching problem of
decomposing the simple $\hlg$-representations $\hV$ as a sum of simple
$\lg$-representations $V$. 
When $\hlg=\lg\times\lg$, it is the tensor product decomposition. 
The multiplicity space $\Mult(V,\hV)$
satisfies
$$
\hV=\oplus_V \Mult(V,\hV)\otimes V,
$$ 
 where the sum runs over the isomorphism classes of simple
 $\lg$-representations.
In the case when $\lg$ is spherical of minimal rank, we describe
$\Mult(V,\hV)$ as the intersection of kernels of powers of root operators in
some weight space of the dual space $V^*$ of $V$.
When $\hlg=\lg\times\lg$, we recover by geometric methods a well known result.
\end{abstract}

\section{Introduction}

Let $G$ be a connected reductive subgroup of a complex semisimple group $\widehat{G}$.
The branching problem consists in 
  decomposing irreducible representations of $\widehat G$ as sum of
  irreducible $G$-representations.

Fix maximal tori $T\subset\hT$ and Borel subgroups $B\supset T$ and
$\hB\supset \hT$ of $G$ and $\hG$ respectively.
Let $X(T)$ denote the group of characters of $T$ and let $X(T)^+$
denote the set of dominant characters.
For $\nu\in X(T)^+$, $V_\nu$ denotes the
irreducible representation of highest weight $\nu$.
Similarly, we use notation $X(\widehat T)$,  $X(\widehat T)^+$,  $V_{\widehat\nu}$
relatively to $\widehat G$. 
For any $G$-representation $V$, the subspace of
 $G$-fixed vectors is denoted by $V^G$.
Given $\nu\in X(T)^+$ and $\widehat\nu\in X(\widehat T)^+$, set
\begin{eqnarray}
  \label{eq:defc}
 \Mult(\nu,\hnu)=\Hom(V_\nu,V_\hnu^*)^G=(V_\nu^*\otimes V_{\widehat\nu}^*)^G,
\end{eqnarray}
where $V_\nu^*$ and $V_\hnu^*$ denote the dual representations of
$V_\nu$ and $V_\hnu$ respectively.
The branching problem is equivalent to the knowledge of these
spaces. Indeed, there is a natural $G$-equivariant isomorphism:
$$
\begin{array}{ccl}
  \bigoplus_{\nu\in X(T)^+} \Hom(V_\nu,V_\hnu^*)^G\otimes
  V_\nu&\longto&V_\hnu^*\\
f\otimes v&\longmapsto&f(v).
\end{array}
$$

Let $\hGB$ denote the complete flag variety of $\hG$.  
In this article, we are interested in the case when the pair $(\hG,G)$
is  {\it spherical of minimal rank}. In other words, we assume that 
there exists $x$ in $\hGB$ such that the orbit $G.x$ of $x$ 
is open in $\hGB$ and the stabilizer $G_x$ of $x$ in $G$ contains a maximal 
torus of $G$. An important example is when $\hG=G\times G$ in which $G$
is diagonally embedded. Then the point $x=(B,B^-)\in\hGB$ works, where
$B^-$ denotes the opposite Borel subgroup of $B$ containing $T$.
More generally, the   spherical pairs of minimal rank have been
classified by the second author in \cite{spherangmin}. The complete
list, assuming in addition that  $\hG$ is semisimple simply connected,
$G$ is 
simple and $G\neq\hG$ is:
\begin{enumerate}
\item \label{list:tensor} $G$ is simple, simply connected and diagonally embedded in $G=G\times G$;
\item \label{list:slsp}
$(\SL_{2n},\Sp_{2n})$ with $n\geq 2$;
\item \label{list:spin}
$(\Spin_{2n},\Spin_{2n-1})$ with $n\geq 4$;

\item \label{list:G2}
$(\Spin_7,G_2)$;
\item \label{list:F4}
$(E_6,F_4)$.
\end{enumerate}

Our aim is to present a uniform  description of the multiplicity spaces
$\Mult(\nu,\hnu)$ for given $(\hG,G)$ in this list. 
Let us first fix some notation. 
Recall that $T\subset\hT$ and denote by $\rho\,:\,X(\hT)\longto X(T)$
the restriction map.
Let $\Delta$ (resp. $\widehat\Delta$) denote the set of simple roots of
$G$ (resp. $\hG$). For each
positive root $\alpha$ of $G$, we fix an $\sl_2$-triple
$(X_\alpha,H_\alpha,X_{-\alpha})$ such that $H_\alpha$ (resp. $X_\alpha$)
belongs to the Lie algebra of $T$ (resp. $B$).
Given $\mu\in X(T)$, we denote by $V_\nu(\mu)=\{v\in V_\nu\,:\,\forall
t\in T\quad tv=\mu(t)v\}$ the corresponding multiplicity space for the
action of the maximal torus $T$.

Let $\hW$ denote the Weyl group of $\hG$.
Fix also $\hy_0\in\hW$ such that $G\hy_0\hB/\hB$ is dense in $\hGB$
and such that $\hy_0$ has minimal length with this
property
(see Section~\ref{sec:choicey0} for details). 
 
For cases \ref{list:slsp} to \ref{list:F4}, we also denote by  $\Phi^1$
the set of long roots of $G$.
In case \ref{list:tensor}, we define $\Phi^1$ to be the empty set.
Set
$$
\Div=\{\ha\in\widehat\Delta\,:\, \rho(\hy_0\ha)\not\in \Phi^1\}.
$$

\begin{theo}
  \label{th:mainintro}
For $\nu\in X(T)^+$ and $\widehat\nu\in X(\widehat T)^+$, there is a natural
isomorphism from $\Mult(\nu,\hnu)$ onto the subspace of
$V_\nu^*(\rho(\hy_0\hnu))$ consisting in the vectors $v$ such that
\begin{enumerate}
\item $X_\alpha\cdot v=0$, for any $\alpha\in \Phi^+\cap\Phi^1$;
\item\label{cond:pole} $X_{-\rho(\hy_0\ha)}^m\cdot v= 0\qquad \forall \, 
  m>\scal{\widehat\nu,\widehat\alpha^\vee}$, for any $\widehat\alpha\in \Div$.
\end{enumerate}
\end{theo}

\bigskip
Actually, the conditions~\ref{cond:pole} are not pairwise
independant. For example, in the case of the tensor product, the
conditions associated to $(\alpha,0)\in \Div$ and $(0,-w_0\alpha)\in \Div$ are
equivalent for any simple root $\alpha$. 
Here $w_0$ denotes the longest element of the Weyl group.  
For each example, we describe in Section~\ref{sec:expl} an explicit
subset of
$\Div$ giving an irredundant set of inequalities. 

In the case of the tensor product, Theorem~\ref{th:mainintro} is well
known. See \cite[Theorem 2.1, p. 392]{PRV} or \cite[Theorem~5, p.~384]{Zelo}.
The usual proofs are algebraic, based on properties of the enveloping
algebra. 
Our geometric proof seems to be new even in this case. 

\bigskip
The branching rule for $(\Spin_{2n},\Spin_{2n-1})$ is multiplicity
free and hence easy to determine (see e.g. \cite{book:FH}).  
That of $(\SL_{2n},\Sp_{2n})$ has been the subject of much
attention in the literature. The first positive rule in terms of dominos was obtained
by Sundaram \cite{Sunda}. Naito-Sagaki conjectured a rule in terms of Littelmann's
patches \cite{NaitoSagaki}. Later, B.~Schumann and J.~Torres proved
this conjecture by obtaining a bijection with Sundaram's model.
A nonpositive rule for $(\Spin_7,G_2)$ were obtained by McGovern in \cite{Mcg:G2}.
Hopefuly, Theorem~\ref{th:mainintro} could be the first step toward
combinatorial rules for these branching problems.

\section{Reminder and complements on spherical homogeneous spaces of
  minimal rank}

\subsection{Roots of $G$ and $\hG$}

Fix a spherical pair of minimal rank $(G,\hG)$ with $G$ and $\hG$ reductive.
Choose a maximal torus $T$ in $G$ and a Borel subgroup $T\subset
B\subset G$. 
Denote by $\Phi$ (resp. $\Phi^+$) the set of roots (resp. positive
roots). 
Recall that $\Delta$ denotes the set of simple roots.
Fix also a maximal torus $\hT$ of $\hG$ containing $T$ and let
$\rho\,:\,X(\hT)\longto X(T)$ denote the restriction map.
The set $\hPhi$ of roots of $\hG$ maps onto $\Phi$ (see
\cite[Lemma 4.2]{spherangmin}). 
Let $\bar\rho$ denote the restriction of $\rho$ to $\hPhi$.
By setting $\hPhi^+=\bar\rho\inv(\Phi^+)$, one gets a
choice of positive roots for $\hG$. Let $\hB$ denote the corresponding
Borel subgroup. 
Then $\hB$ contains $B$.
By \cite[Lemma~4.6]{spherangmin}, $\rho(\widehat\Delta)=\Delta$, where
$\widehat\Delta$ denote the set of simple roots of $\hG$. 

On the following diagrams the restriction of $\bar\rho$ to the set of
simple roots is the vertical projection.

\begin{center}
  \begin{tabular}{|c|c|c|}
    \hline
    \begin{tikzpicture}[scale=.9,baseline=(2.center)]
      \draw (0,0) node[anchor=west] {$A_{2n-1}$};
      \node (vide) at (1,1) {};
      \node[dnode] (2) at (2,0) {}; \node[dnode] (3) at (3,-0.7) {};
      \node[dnode] (4) at (3,0.7) {}; \node[dnode] (5) at (4,-0.7) {};
      \node[dnode] (6) at (4,0.7) {}; \node[dnode] (7) at (5,-0.7) {};
      \node[dnode] (8) at (5,0.7) {}; \node[dnode] (9) at (6,-0.7) {};
      \node[dnode] (0) at (6,0.7) {};

      \path (2) edge[sedge] (3) edge[sedge] (4); \path (5) edge[sedge]
      (3) edge[sedge] (7);\path (6) edge[sedge] (4) edge[sedge] (8);
      \path (7) edge[sedge] (9); \path (8) edge[sedge] (0);


      \draw (0,-1.7) node[anchor=west] {$C_n$};
      \node[dnode] (3) at (2,-1.7) {}; \node[dnode] (4) at (3,-1.7)
      {}; \node[dnode] (5) at (4,-1.7) {}; \node[dnode] (6) at
      (5,-1.7) {}; \node[dnode] (7) at (6,-1.7) {};

   
      \path (3) edge[dedge] (4); \path (4) edge[sedge] (5); \path (6)
       edge[sedge] (5) edge[sedge] (7);
     \end{tikzpicture}&
                        \begin{tikzpicture}[scale=.9,baseline=(1.center)]
                          \draw (-0.5,0) node[anchor=west] {$D_n$};
                          \node[dnode] (1) at (1,0) {}; \node[dnode]
                          (2) at (2,0) {}; \node[dnode] (3) at (3,0)
                          {}; \node[dnode] (4) at (4,0) {};
                          \node[dnode] (5) at (5,-0.7) {};
                          \node[dnode] (6) at (5,0.7) {};

                          \path (1) edge[sedge] (2); \path (3)
                           edge[sedge] (2) edge[sedge] (4); \path (4)
                            edge[sedge] (5) edge[sedge] (6);


                            \draw (-0.5,-1.7) node[anchor=west]
                            {$B_{n-1}$}; \node[dnode] (2) at (1,-1.7)
                            {}; \node[dnode] (3) at (2,-1.7) {};
                            \node[dnode] (4) at (3,-1.7) {};
                            \node[dnode] (5) at (4,-1.7) {};
                            \node[dnode] (6) at (5,-1.7) {};

                            \path (2) edge[sedge] (3); \path (3)
                             edge[sedge] (4); \path (5) edge[dedge]
                             (6); \path (4) edge[sedge] (5);
                           \end{tikzpicture}

    \\
    \hline
    \begin{tikzpicture}[scale=.9,baseline=(1.center)]
      \node (vide) at (1,1) {};
      \draw (-0.5,0) node[anchor=west] {$E_6$};
      \node[dnode] (1) at (1,0) {}; \node[dnode] (2) at (2,0) {};
      \node[dnode] (3) at (3,-0.7) {}; \node[dnode] (4) at (3,0.7) {};
      \node[dnode] (5) at (4,-0.7) {}; \node[dnode] (6) at (4,0.7) {};

      \path (1) edge[sedge] (2); \path (2) edge[sedge] (3) edge[sedge]
      (4); \path (5) edge[sedge] (3);\path (6) edge[sedge] (4);


      \draw (-0.5,-1.7) node[anchor=west] {$\tilde{F}_4$};
      \node[dnode] (2) at (1,-1.7) {}; \node[dnode] (3) at (2,-1.7)
      {}; \node[dnode] (4) at (3,-1.7) {}; \node[dnode] (5) at
      (4,-1.7) {};
    
      \path (2) edge[sedge] (3);
   
      \path (3) edge[dedge] (4); \path (4) edge[sedge] (5);
    \end{tikzpicture}&
                       \begin{tikzpicture} [scale=.9,baseline=(1.center)]
                         \draw (0.5,0) node[anchor=west] {$B_3$};
                         \node[dnode] (2) at (2,0) {}; \node[dnode]
                         (3) at (3,-0.7) {}; \node[dnode] (4) at
                         (3,0.7) {}; \path (2) edge[dedge] (3)
                         edge[sedge] (4);

                         \draw (0.5,-1.7) node[anchor=west] {$G_2$};
                         \node[dnode] (2) at (2,-1.7) {}; \node[dnode]
                         (3) at (3,-1.7) {};

                         \path (2) edge[tedge] (3);
                       \end{tikzpicture}\\
    \hline
  \end{tabular}
\end{center}

By \cite[Lemma~4.4]{spherangmin}, for any $\alpha\in\Phi$,
$\bar\rho\inv(\alpha)$ has cardinality one or two. Hence, $\Phi$
splits in the following two subsets

\begin{center}
\begin{tabular*}{0.8\textwidth}{{@{\extracolsep{\fill}}ccc}}
$
\Phi^1:=\{\alpha\in\phi_\lh\;:\;\sharp\bar\rho^{-1}(\alpha)=1\}
$ &and &
$
\Phi^2:=\{\alpha\in\phi_\lh\;:\;\sharp\bar\rho^{-1}(\alpha)=2\}.
$
\end{tabular*}
\end{center}
Set also $\hPhi^1=\bar\rho\inv(\Phi^1)$ and
$\hPhi^2=\bar\rho\inv(\Phi^2)$.
It is easy to check using the pictures, that this set $\Phi^1$, coincides
with that defined in the introduction. 

\subsection{Isotropies}

Assume, in addition, that $G$ is connected.

  \begin{lemma}
\label{lem:isot}
  Let $x\in \hGB$. Then  the isotropy $G_x$ is connected and contains
  a maximal torus of $G$.
  \end{lemma}

\begin{proof}
  The fact that $G_x$ contains a maximal torus $T$ of $G$
 follows from the monotonicity properties of the rank of orbit closures
  of $G$ in $\hG/\hB$ (see \cite[Theorem~2.2]{Kn:WGH}).

  Let $L$ be a Levi subgroup (maximal reductive subgroup) of $G_x$.
  Then $L$ is isomorphic to the quotient of $G_x$ by its
  unipotent radical.  
A  reference for the existence of $L$ is \cite[Theorem~4. p286]{OV:book}. 
Up to conjugacy, one may assume that  $T\subset L$.

  But $L$ is a reductive subgroup of $\hG_x=\hB'$, that is a Borel
  subgroup of $\hG$. Hence $L$ maps injectively into
  $\hB'/\widehat U'\simeq \hT$ and $L$ is abelian. The torus
  $T$ being its own centralizer in $G$, we deduce
  that $L=T$.

  Now $G_x$ is connected as the product of $T$ and its unipotent radical.
\end{proof}

\subsection{Orbits of $G$ in $\hGB$}
\label{sec:choicey0}

We are now interested in the set ${G\backslash}\hGB$ of $G$-orbits
in $\hGB$. Let $W$ and $\hW$ be the Weyl groups of $G$ and $\hG$
respectively.
Since $T$ is a regular torus in $\hG$ (see \cite[Lemma~2.3]{Br:ratsmooth}), $W$ naturally embeds
in $\hW$.

\begin{lemma}
\label{lem:paramGorb}
  The map
$$
\begin{array}{ccl}
{W\backslash}\hW&\longto&{G\backslash}\hGB\\
\widehat w&\longmapsto&G\widehat w\widehat B/\widehat B
\end{array}
$$
is a well-defined bijection.
\end{lemma}

\begin{proof}
It is clear that the $G$-orbit   $G\widehat w\widehat B/\widehat B$ does not
depend on the representative $\widehat w$ in its class $W\widehat w$. 
Hence, the map of the lemma is well-defined.

Since  $T$ is a regular torus in $\hG$, its fixed points in $\hG/\hB$
are the points $\hw\hB/\hB$ for $\widehat w\in\hW$.
But Lemma~\ref{lem:isot} implies that each $G$-orbit in $\hG/\hB$
contains a $T$-fixed point. The surjectivity follows.

Let now $\hw$ and $\hw'$ in $\hW$ such that $G\widehat w\widehat B/\widehat
B=G\widehat w'\widehat B/\widehat B$.
To get the injectivity and finish the proof, it remains to prove that
$W\widehat w=W\widehat w'$.
Choose $g\in G$ such that $g \hw\hB/\hB=\hw'\hB/\hB$.
Let $H$ and $H'$ denote the isotropies in $G$ of the points
$\hw\hB/\hB$ and $\hw'\hB/\hB$, respectively.
Observe that $T$ is a maximal torus of both $H$ and $H'$. 
Then, $T$ and $gTg\inv$ are two maximal tori of $H'$, and 
there exists $h'\in H'$ such that $h'Th'\inv=gTg\inv$.
Thus, $n:=h'\inv g$ normalizes $T$ and satisfies $
n \hw\hB/\hB=h'\inv \hw'\hB/\hB=\hw'\hB/\hB$. It follows that $\hw'\in W\widehat w$.
\end{proof}

Let $\ell\,:\,\hW\longto \NN$ denote the length function. 
It is well known that $\ell(\hw)=\sharp(\hPhi^+\cap \hw\inv\hPhi^-)$,
where
$\hPhi^-=-\hPhi^+$.

We fix, once for all, $\hy_0\in\hW$ such that $G\hy_0\hB/\hB$ is the
open $G$-orbit in $\hGB$ and such that $\hy_0$ has minimal length with this
property.
Let $H_0$ denote the stabilizer of $\hy_0\hB/\hB$ in $G$.

Given a root $\alpha$ (resp. $\halpha$) of $\lg$ (resp. $\hlg$),
denote by $\lg_\alpha$ (resp. $\lg_\halpha$) the corresponding root
space.

\begin{lemma}
\label{lem:Lieisoty0}
The Lie algebra of $H_0$ is 
$$
\Lie(T)\oplus\bigoplus_{\alpha\in \Phi^+\cap\Phi^1}\lg_\alpha.
$$  
\end{lemma}

\begin{proof}
It is clear that $H_0$ contains $\hT\cap G=T$. Then
$$
\Lie(H_0)=\Lie(T)\oplus\bigoplus_{\alpha\in \Sc}\lg_\alpha,
$$
for some subset $\Sc$ of $\Phi$.

Let now $\alpha\in \Phi^1\cap\Phi^+$. We want to prove that
$\alpha\in\Sc$.
Let $\widehat\alpha\in\hPhi$ such that $\rho(\widehat\alpha)=\alpha$. 
We have $\lg_{\pm\widehat\alpha}=\lg_{\pm \alpha}$.
But, exactly one between  $\lg_{\widehat\alpha}$ and $\lg_{-\widehat\alpha}$
is contained in the Borel Lie algebra $\Lie(\hy_0\hB\hy_0\inv)$.
Assume, for a contradiction that $\lg_{-\widehat\alpha}\subset\Lie(\hy_0\hB\hy_0\inv)$.

Set $\widehat\beta=-\hy_0\inv\widehat\alpha$ and
$\widehat{y}_0'=s_{\widehat\alpha}\hy_0$. 
Since $\widehat\beta$ is positive and $\hy_0\widehat\beta$ is negative,
$\ell(\widehat{y}_0')<\ell(\hy_0)$.
But, $\lg_{\pm\widehat\alpha}=\lg_{\pm \alpha}$ implies that
$s_\alpha=s_{\widehat\alpha}$ and that $\widehat{y}_0'\in W\hy_0$. 
This contradicts the minimality assumption on the length of
$\hy_0$.\\

At this point, we proved that the Lie algebra of the lemma is
contained in $\Lie(H_0)$. But $\dim(G/H_0)=\dim(\hGB)$, hence
$$
\begin{array}{ll}
  \dim(H_0)&=\dim(G)-\dim(\hGB)\\
&=\dim(T)+2\sharp\Phi^+-\sharp\hPhi^+\\
&=\dim(T)+\sharp(\Phi^+\cap\Phi^1)
\end{array}
$$
and we can conclude.
\end{proof}

\begin{lemma}
 \label{lem:b+b-}
Fix $\beta \, \in \Phi^2$. Then there is exactly one root in $\bar\rho \inv (\beta)$ which is sent in $\hPhi^+$ by the action of $\hy_0 \inv$. 
 \end{lemma}
 
\begin{proof}
 Write $\bar\rho \inv (\beta) = \{ \widehat\beta_1 \, , \widehat\beta_2 \}$. Because of \cite[Lemma~4.4]{spherangmin}, we have that $ \lg_{\beta} \subseteq \lg_{\widehat\beta_1}\oplus \lg_{\widehat\beta_2}$ and  $ \lg_{\beta} \neq   \lg_{\widehat\beta_i}$ for $i=1 \, , 2.$\\
  Lemma~\ref{lem:Lieisoty0} implies that  $ \lg_{\beta} \not \subseteq \Lie{(H_0)}$. But since $H_0= H \,  \cap \, \hy_0\hB \hy_0 \inv$, this means that $\hy_0 \inv \beta_1$ and $\hy_0 \inv \widehat\beta_2 $ are not both positive roots, equivalently: $ \{ \hy_0 \inv \widehat\beta_1 \, , \, \hy_0 \inv \widehat\beta_2 \} \not \subseteq \hPhi^+$.
  The same argument applied to $- \beta$ implies that $ \{- \hy_0 \inv
  \widehat\beta_1 \, , \, -\hy_0 \inv \widehat\beta_2 \} \not \subseteq
  \hPhi^+$. The lemma follows.
\end{proof}
 
\bigskip
 Lemma \ref{lem:b+b-} allows us to distinguish between the roots in the fiber of $\bar\rho$ by means of the action of $ \hy_0 \inv$. In particular we introduce the following notation. 
 If $\beta \in \Phi^2 \, \cap \, \Phi^+$, then $\widehat\beta^+ $ (resp. $\widehat\beta^-$ ) is the unique element in $\bar\rho \inv (\beta) $ that satisfies: $\hy_0 \inv \widehat\beta^+ \in \hPhi^+$ (resp. $\hy_0 \inv \widehat\beta^- \in \hPhi^-$).

\subsection{The graph of $G$-orbits in $\hGB$}

We recall the definition given in \cite{GammaGH} of a graph $\Gamma(\hG/G)$
whose vertices are the elements of $G\backslash\hGB$. The original
construction of $\Gamma(\hG/G)$ due to M. Brion is equivalent but 
slightly different (see \cite{Br:GammaGH}).

If $\widehat \alpha$ belongs to $\widehat\Delta$, we denote by $P_\ha$ the associated
minimal standard parabolic subgroup of $\hG$.
Consider the unique $\hG$-equivariant map
$\pi_\ha\,:\,\hGB\longto\hG/P_\ha$ which is a $\PP^1$-bundle. 

Let $\Orb\in G\backslash\hGB$ and $\ha\in\widehat\Delta$. 
Consider $\pi_\ha\inv(\pi_\ha(\Orb))$.  
Two cases occur.
\begin{itemize}
\item $G$ acts transitively on $\pi_\ha\inv(\pi_\ha(\Orb))$.
\item $\pi_\ha\inv(\pi_\ha(\Orb)$ contains two $G$-orbits, one
  closed $V$ and one open $V'$. Then we says that {\it $\widehat\alpha$
    raises $V$ to $V'$}. In this case, $\dim(V')= \dim(V)+1$, and for any $x \in V$, $\pi_\ha\inv(\pi_\ha(x)) \cap V = \{x \}$.
\end{itemize}

\begin{defi}
Let $\Gamma(\hG/G)$ be the oriented graph with vertices the elements of
$G\backslash\hGB$ and edges labeled by $\widehat\Delta$, where $V$ is joined to $V'$
by an edge labeled by $\ha$ if $\ha$ raises $V$ to $V'$. 
\end{defi}

\begin{lemma}
\label{lem:ellm}
  Let $\hw\in \hW$.
We have 
$$
\dim G\hw\hB/\hB-\dim G/B\leq \ell(\hw).
$$
Moreover, for any $G$-orbit $\Orb$ in $\hGB$ there exists $\hw\in \hW$
such that $\Orb=G\hw$ and
$$
\dim G\hw\hB/\hB-\dim G/B=\ell(\hw).
$$
\end{lemma}

\begin{proof}
  We consider the $\hB$-orbits in $\hG/G$.
The closed orbit is $\hB/B$. 
Choose a reduced expression of $\hw=s_{\ha_s}\cdots
s_{\ha_1}$. Consider the quotient  $P_{\halpha_1}\times_\hB \cdots \times_\hB P_{\halpha_s}\times_\hB
\hB/B$  of $P_{\halpha_1}\times \cdots \times P_{\halpha_s}\times
\hB/B$ by the action of $B^s$ given by 
$(b_1,\dots,b_s).(p_1,\dots,p_s,\hB/B)=(p_1b_1\inv,\dots,b_{s-1}p_sb_s^\inv,b_s\hB/B)$
(with obvious notation). Consider the regular map
$$
\begin{array}{ccl}
P_{\halpha_1}\times_\hB \cdots \times_\hB P_{\halpha_s}\times_\hB
\hB/B &\longto& \hG/G\\
(p_1,\dots,p_s,x)&\longmapsto&p_1\cdots p_sx.
\end{array}
$$
The dimension of the left hand side is $\ell(\hw)+\dim(\hB/B)$. 
The right hand side is a $\hB$-orbit closure containing $\hB\hw\inv
G/G$.
The first inequality follows.

This equality is reached when the expression is obtained by reading
the labels on some path from the closed orbit to $\Orb$ in the graph
$\Gamma(\hG/G)$.
Such a path exists by \cite[Proposition~2.2]{spherangmin}. 
\end{proof}

\bigskip
Set
$$
\ell_m(\hw)=\min\{\ell(w\hw)\,:\, w\in W\}.
$$
By Lemma~\ref{lem:ellm}, $\ell_m(\hw)=\dim (G\hw\hB/\hB)-\dim(G/B)$. If an element in $ \hw \in \hW$ satisfies $\ell(\hw)= \ell_m(\hw)$, \, we say that $\hw$ has \textit{minimal length}.



\begin{lemma}
\label{lem:carddiv}
\begin{enumerate}
\item There are $\sharp \Delta^2$ codimension one $G$-orbits in
  $\hGB$.
\item Let $\widehat\alpha\in\widehat\Delta$. Then $G\hy_0s_\halpha\hB/\hB$ has
  codimension one if and only if $\hy_0\halpha\in\hPhi^2$.
\item Let $\widehat\alpha\in\widehat\Delta$ such that
  $\hy_0\halpha\in\hPhi^2$. Then, the Lie algebra of the stabilizer of
  $\hy_0s_\halpha\hB/\hB$ is
$$
\Lie(H_0)\oplus \lg_{-\rho(\hy_0\halpha)}.
$$
\end{enumerate}
\end{lemma}

\begin{proof}
  By \cite[Proposition~2.3]{spherangmin}, the number of codimension one $G$-orbits in $\hGB$
  is the number of $G$-orbits $\Orb$ of dimension $\dim(G/B)+1$.
This is also 
$$
\sharp(\{Ws_\halpha\in\; W\setminus\hW \, ; \, \halpha\in\widehat\Delta \}-\{W\}).
$$
But $s_\halpha\in W$ if and only if $\widehat\alpha\in\widehat\Delta^1$. 
Moreover, if  $\widehat\alpha_1\neq \widehat\alpha_2 \in\widehat\Delta^2$  then
$s_{\halpha_1}s_{\halpha_2}\in W$ if and only if $\rho(
\widehat\alpha_1)=\rho(\widehat\alpha_2)$.\\

Since $G\hy_0\hB/\hB$ is dense, either
$G\hy_0s_\halpha\hB/\hB=G\hy_0\hB/\hB$ or $G\hy_0s_\halpha\hB/\hB$ has codimension one. 
But $\hy_0s_\halpha\hB/\hB$ belongs to the open $G$-orbit if and only
if $\hy_0s_\halpha\in W \hy_0$ if and only
if $\hy_0\halpha\in\hPhi^1$.\\

Let $\ha$ be like in the last assertion. 
Then
$$
\hy_0s_\ha\hPhi^+=(\hy_0\hPhi^+\cup\{-\hy_0\ha\})-\{\hy_0\ha\}.
$$
Now,
the fact that $\Lie(H_0)$ is contained in the  Lie algebra of the stabilizer of
  $\hy_0s_\halpha\hB/\hB$ follows from Lemma~\ref{lem:Lieisoty0}.

The root $\hy_0\halpha$ belonging to $\hPhi^2$, there exists a second
root $\widehat\beta$ in the fiber $\bar\rho\inv(\rho(-\hy_0\halpha))$.
But, Lemma~\ref{lem:b+b-} implies that $\hbeta$ belongs to $\hy_0\hPhi^+$ and
hence to $\hy_0s_\ha\hPhi^+$.
We deduce that
$$
\lg_{\rho(-\hy_0\halpha)}\subset \lg_{\widehat\beta}\oplus
\lg_{-\hy_0\ha}\subset\Lie(\hy_0s_\ha B (\hy_0s_\ha)\inv).
$$
In particular, $\lg_{\rho(-\hy_0\halpha)}$ is contained in the Lie algebra of the stabilizer of
  $\hy_0s_\halpha\hB/\hB$. 

Now, the last assertion follows by equality of the dimensions.
\end{proof}

\subsection{Compatible $\sl_2$-triples}

Recall that we have fixed an $\sl_2$-triple
$(X_\beta,H_\beta,Y_\beta)$ associated to any positive
root $\beta$ of $G$.

\begin{lemma}
\label{sltriples}
If $\beta \in \Phi^2 \cap \Phi^+$ then there exist $\sl_2$-triples  $( X_{\hbeta^+}, H_{\hbeta^+}, X_{- \hbeta^+} )$ and  $( X_{\hbeta^-}, H_{\hbeta^-}, X_{-\hbeta^-})$  for $\hbeta^+$ and $\hbeta^-$ such that:\\
\begin{itemize}
    \item $X_{\beta}=X_{\beta^+}+X_{\beta^-}; \hspace{0.5cm} H_{\beta}=H_{\beta^+}+H_{\beta^-}; \hspace{0.5 cm} X_{-\beta}=X_{-\beta+}+X_{-\beta^-}$.
    \item $ \exp(tX_{\beta}) = \exp(tX_{\hbeta^+}) \cdot \exp(tX_{\hbeta^-})
    $ for any $t \in \CC$
\end{itemize}

\end{lemma}

\begin{proof}
Since $\lg_{\beta} \subseteq \lg_{\hbeta^+} \oplus \lg_{\hbeta^-}$ and
is not equal to any of the 2 root spaces on the right hand side,
$X_{\beta}=X_{\hbeta^+}+X_{\hbeta^-}$ for nonzero  $
X_{\hbeta^\pm}\in\lg_{\hbeta^\pm}$. Then there exist $\sl_2$-triples
of the form  $( X_{\hbeta^+}, H_{\hbeta^+}, X_{- \hbeta^+} )$ and $( X_{\hbeta^-}, H_{\hbeta^-}, X_{-\hbeta^-})$ \, for $\hbeta^+$ and $\hbeta^-$.
Moreover, $X_{-\beta}= a X_{-\hbeta^+} + b X_{-\hbeta^-}$ with nonzero
constants $a$ and $b$, and 
$$H_{\beta}=[X_{\beta}, X_{-\beta}]=[ X_{\hbeta^+}+X_{\hbeta^-} ,  a X_{-\hbeta^+} + b X_{-\hbeta^-}]= aH_{\hbeta^+} + bH_{\hbeta^-} .$$
For the last equality we use that $[X_{\hbeta^+}, X_{-\hbeta^-}]=0$ and $[X_{\hbeta^-}, X_{-\hbeta^+}]=0$. 
Indeed, if $\hbeta^+ -\hbeta^-$ (reps. $\hbeta^- -\hbeta^+$) was a root in $\hPhi$, then $\rho(\hbeta^+ -\hbeta^-)=0$ (resp. $\rho(\hbeta^- -\hbeta^+)=0$). But this is absurd since $\rho$ sends $\hPhi$ in $\Phi$.\\
Using that $\rho(\hbeta^+)=\rho(\hbeta^-)=\beta$, we get that:
$$ 2 = \beta( H_ {\beta})= \hbeta^+ ( aH_{\hbeta^+} + bH_{\hbeta^-})= a \hbeta^+ ( aH_{\hbeta^+}) =2a  .$$
Where for the second-last inqueality we used that $\hbeta^+$ and $\hbeta^-$ are ortogonal, hence $\hbeta^+(H_{\hbeta^-})=0$. Then $a=1$, and similarly we get $b=1$. This proves the first point.\\
\\
For the second one notice that $[X_{\hbeta^+} , X_{\hbeta^-}]=0$. In fact if $\hbeta^+ + \hbeta^-$ was a root, then $\rho(\hbeta^+ + \hbeta^-)=2 \beta$. But this is absurd since $2 \beta$ is not a root. Hence for any $t \in \CC$,  $\exp(tX_{ \hbeta^+} + tX_{ \hbeta^-})=\exp(tX_{\hbeta^+}) \cdot \exp(tX_{\hbeta^-})$, and the second point follows immediately.
\end{proof}
\begin{remark}
For the proof we could also used \cite[Lemma 4.1]{spherangmin} and \cite[Lemma4.3]{spherangmin} to reduce the problem to the case of $\PSL_2$ diagonally embedded in $\PSL_2 \times \PSL_2$.
\end{remark}

\section{Proof of Theorem~\ref{th:mainintro}}

\subsection{An embedding of the multiplicity space}

Fix dominant weights $\nu\in X(T)^+$ and $\widehat\nu\in X(\hT)^+$.
Observe that $H_0$ is a subgroup of $\hy_0\hB\hy_0\inv$ and $\hy_0\hnu$
is a character of this last group. 
In particular, $\hy_0\hnu$ restricts as a character of $H_0$.

\begin{lemma}
\label{H^0-seminvariants}
  In the $G$-representation $V_\nu^*$, the subspaces
$$
 (V_\nu^*)^{(H_0)_{\hy_0\hnu}} = \{v\in V_\nu^*\,:\, \forall h\in H_0\quad h.v=(\hy_0\hnu)(h)v\}
$$
and 
$$
\{v\in V_\nu^*(\hy_0\hnu)\,:\, \forall \alpha\in
  \Phi^+\cap\Phi^1 \quad X_\alpha\cdot v=0\}
$$
coincide.
\end{lemma}

\begin{proof}
 Since $T$ is a maximal torus of $H_0$,  the first subspace is contained in
  $V_\nu^*(\hy_0\hnu)$.
Furthermore $H_0=R^u(H_0)T$, where $R^u(H_0)$ denotes the unipotent
radical. Hence
$$
 (V_\nu^*)^{(H_0)_{\hy_0\hnu}}=V_\nu^*(\hy_0\hnu)^{R^u(H_0)}=
 V_\nu^*(\hy_0\hnu)^{\Lie(R^u(H))}.
$$
Now, Lemma~\ref{lem:Lieisoty0} allows to conclude.
\end{proof}

\bigskip
Consider on $G/B$  the $G$-linearized line bundle $\Li_\nu$ such that $B$ acts
on the fiber over $B/B$ with weight $-\nu$.
Similarly, define $\Li_\hnu$.
By the Borel-Weyl theorem, the space
$$
H^0(G/B\times\hGB,\Li_\nu\otimes\Li_\hnu)^G\simeq (V_\nu^*\otimes V_\hnu^*)^G
$$
of $G$-invariant sections identifies with $\Mult(\nu,\hnu)$.
The orbit $G.\hy_0\hB/\hB$ being open, the restriction map
$$
H^0(G/B\times\hGB,\Li_\nu\otimes\Li_\hnu)^G\longto H^0(G/B\times
G.\hy_0\hB/\hB,\Li_\nu\otimes\Li_\hnu)^G
$$
is injective. Moreover
$$
H^0(G/B\times G.\hy_0\hB/\hB,\Li_\nu\otimes\Li_\hnu)^G \simeq
H^0(G/B,\Li_\nu)^{(H_0)_{\hy_0\hnu}}
\simeq (V_\nu^*)^{(H_0)_{\hy_0\hnu}}.
$$
For $\varphi\in (V_\nu^*)^{(H_0)_{\hy_0\hnu}}$ and denote by $\tilde\sigma\in
H^0(G/B\times G.\hy_0\hB/\hB,\Li_\nu\otimes\Li_\hnu)^G$ the associated
section. 
To describe the image of
$H^0(G/B\times\hGB,\Li_\nu\otimes\Li_\hnu)^G$ in
$(V_\nu^*)^{(H_0)_{\hy_0\hnu}}$, we have to understand what sections
$\tilde\sigma$ extend, and hence the order of the poles of $\tilde\sigma$
along the divisors of $(G/B\times\hGB)-(G/B\times G.\hy_0\hB/\hB)$.

\subsection{Local description along the divisors}

Fix $\widehat\alpha\in \widehat\Delta$ that is a label of some edge descending
from the open orbit in $\Gamma(\hG/G)$. Then $ D_{\ha} := \overline{G \hy_0
  s_\halpha\hB/\hB}$ is a divisor of $\hGB$. By Lemma~\ref{lem:carddiv} this happens if and only if $\hy_0 s_\halpha \in \hPhi^2$. 
Set $\beta=\pm\rho(\hy_0 \halpha)$ the sign being chosen to get $\beta\in\Phi^+$.

Lemma \ref{lem:b+b-} allows to distinguish two situations that  lead
to two slightly different local descriptions of the corresponding
divisor.

\begin{enumerate}
\item[1)] $\hy_0\widehat\alpha$ is negative. 
Then $\beta=-\rho(\hy_0\widehat\alpha)$ and $\hy_0\widehat\alpha=-\hbeta^-$.

\item[2)] $\hy_0\widehat\alpha$ is positive.
Then $\beta=\rho(\hy_0\widehat\alpha)$ and $\hy_0\widehat\alpha=\hbeta^+$.

\end{enumerate}

\bigskip
Given a root $\alpha$ (resp. $\halpha$) of $G$ (resp. $\hG$), we
denote by $U_\alpha$ (resp. $U_\halpha$) the corresponding subgroup
isomorphic to $(\CC,+)$.
The aim of this section is to describe an
open subset of $\hG/\hB$ intersecting the divisor $D_{\ha}$. In particular we will
prove that $U_{\ha}$ is a common transverse slice to $D_ {\ha}$ at any
point of this subset. Before doing so we need some preparatory work.

\begin{lemma}
\label{geomorbits}
Let $\ha$,  $\beta$ and $D_{\ha}$ be as above. Set $ \Sc(\ha)
:=(\Phi^+ \cap \Phi^2) \setminus \{\beta \}$. 
Index the elements of $ \Sc(\ha)=\{\gamma_1,\dots,\gamma_s\}$.
Then: 
\begin{enumerate}
    \item [1)] If $\hy_0\widehat\alpha$ is negative then  the map
    
   $$
   \begin{array}{cccl}
    i_{\ha}\,: &U^- \times \prod_{\gamma \in \Sc(\ha)} U_{\gamma}
     &\longto&G {\hy_0s_{\ha} \hB / \hB }\\
&(u^-, (u_{\gamma})_{\gamma} )&\longmapsto&(u^-  \displaystyle \prod_{i=1}^s u_{\gamma_i})  \, \hy_0  \, s_{\ha} \hB/ \hB
   \end{array}
$$

is an open immersion. 

\item [2)] If $\hy_0\widehat\alpha$ is positive then  the map
    
   $$
   \begin{array}{cccl}
    i_{\ha}\,: &U^- \times \prod_{\gamma \in \Sc(\ha)} U_{\gamma}
     &\longto&G {\hy_0s_{\ha} \hB / \hB }\\
&(u^-, (u_{\gamma})_{\gamma} )&\longmapsto&s_\beta(u^-  \displaystyle \prod_{i=1}^s u_{\gamma_i}) s_\beta \, \hy_0  \, s_{\ha} \hB/ \hB
   \end{array}
$$

    

is an open immersion.

\end{enumerate}
\end{lemma}

Our proof of Lemma~\ref{geomorbits} depends on the following
well-known lemma.

\begin{lemma}\label{lem:unip}
  \begin{enumerate}
  \item The product induces an open immersion
$$
\begin{array}{ccl}
  U^-\times U\times T&\longto&G\\
(u^-,u,t)&\longmapsto&u^-ut.
\end{array}
$$
\item Let $H$ be a $T$-stable closed subgroup of $U$. Number
  $\gamma_1,\dots,\gamma_s$ the positive roots of $G$ that are not
  roots of $H$. Then the map
$$
\begin{array}{ccl}
  U_{\gamma_1}\times\cdots\times U_{\gamma_s} \times H&\longto&U\\
((u_i),u)&\longmapsto&u_1\dots u_s\,u
\end{array}
$$
is an isomorphism.
  \end{enumerate}
\end{lemma}

\begin{proof}
  The first assertion is a part of the Bruhat decomposition (see
  e.g. \cite[Section~28.5]{Hum}). The second assertion is the content
  of \cite[Section~28.1]{Hum}.
\end{proof}

\bigskip
    
\begin{proof}We begin by the negative case. Notice that Lemma~\ref{lem:carddiv} implies that $\gamma \in \Sc(\ha)$ if and only if $U_\gamma \not \subseteq U_{\hy_0 s_\ha \hB/ \hB}$.
Applying the second assertion of Lemma~\ref{lem:unip} with $H=U_{\hy_0
  \, s_{\ha} \hB/ \hB}$, one gets an isomorphism 
$$
\prod_{\gamma \in \Sc(\ha)} U_{\gamma}\times U_{\hy_0
  \, s_{\ha} \hB/ \hB}\longto U.
$$
Now, again by Lemma~\ref{lem:carddiv}, $G_{\hy_0
  \, s_{\ha} \hB/ \hB}=T U_{\hy_0
  \, s_{\ha} \hB/ \hB}
$. Then the first assertion of Lemma~\ref{lem:unip} implies that 
$$
U^-\times \prod_{\gamma \in \Sc(\ha)} U_{\gamma}\times G_{\hy_0
  \, s_{\ha} \hB/ \hB}\longto G
$$
is an open embedding.
Taking the quotient by $G_{\hy_0
  \, s_{\ha} \hB/ \hB}$ we conclude.\\

The second case is obtained similarly by applying Lemma~\ref{lem:unip}
to the maximal unipotent subgroups $s_\beta U^\pm s_\beta$ in place of $U^\pm$. 
\end{proof}

\bigskip
\begin{prop}
\label{charts}
Keep the  setting as in Lemma~\ref{geomorbits}. 
\begin{enumerate}
    \item [1)] If $\hy_0\widehat\alpha$ is negative then  the map:
    
   $$\begin{array}{cccl}
  f_{\ha} \,:&U^- \times \prod_{\gamma \in \Sc(\ha)} U_{\gamma} \times
               U_{-\ha}&\longto&\hG/\hB\\
&(u^-, (u_{\gamma})_{\gamma} , u_{-\ha})&\longmapsto&(u^-  \displaystyle \prod_{i=1}^s u_{\gamma_i})  \, \hy_0 s_{\ha}  u_{-\ha} \,  \hB  
\end{array}
$$
is an open immersion. 
Furthermore the image of $f_\ha$ is contained in $G \hy_0 \hB/ \hB \cup D_\ha$ , and\, $ f_{\ha} \inv (D_{\ha})= U^-\times \prod_{\gamma \in \Sc(\ha)} U_{\gamma} \times\{1 \}$.  

 \item [2)] If $\hy_0\widehat\alpha$ is positive then  the map:
    
   $$\begin{array}{cccl}
  f_{\ha} \,:&U^- \times \prod_{\gamma \in \Sc(\ha)} U_{\gamma} \times
               U_{-\ha}&\longto&\hG/\hB\\
&(u^-, (u_{\gamma})_{\gamma} , u_{-\ha})&\longmapsto&s_\beta(u^-  \displaystyle \prod_{i=1}^s u_{\gamma_i
})  \, s_\beta \hy_0s_\ha u_{-\ha} \,  \hB  
\end{array}
$$
is an open immersion. 
Furthermore the image of $f_\ha$ is contained in $G \hy_0 \hB/ \hB \cup D_\ha$, and \, $ f_{\ha} \inv (D_{\ha})= U^-\times \prod_{\gamma \in \Sc(\ha)} U_{\gamma} \times\{1 \}$.  
    
\end{enumerate}
\end{prop}

\smallskip
\begin{proof} We prove part 1). The proof of the positive case is obtained from the following one by replacing every appearance of $\hy_0$ with $s_{\beta} \hy_0$.
Identify $\hGB$ with the fibered product $\hG \times_{P_\ha} P_\ha / \hB$. Via this identification $f_\ha$ is the map:
 $$\begin{array}{cccl}
  f_{\ha} \,:&U^- \times \prod_{\gamma \in \Sc(\ha)} U_{\gamma} \times
               U_{-\ha}&\longto&\hG \times_{P_\ha} P_\ha / \hB\\
&(u^-, (u_{\gamma})_{\gamma} , u_{-\ha})&\longmapsto&\bigl[ (u^- \displaystyle  \prod_{i=1}^s u_{\gamma_i} \hy_0 s_{\ha} \, : \, u_{-\ha} \hB / \hB)  \bigr].
\end{array}
$$
Since $ \hy_0 s_\ha \hB/ \hB $  and $ \pi_\ha(\hy_0 s_\ha \hB/ \hB)$
have the same isotropy in $G$, we can identify their $G$-orbits.
In particular we will think about $i_\ha$, the open immersion of Lemma~\ref{geomorbits}, as a map to $G.\hy_0 s_{\ha}P_\halpha /P_\halpha \subseteq \hG / P_\ha$. Call $V_{\ha} \subseteq \hG/ P_\ha$ the image of $i_{\ha}$. By Lemma~\ref{geomorbits}, $V_{\ha}$ is open in $\hG / P_\ha$.
Denote the two components of 
$$
  i_{\ha} \inv : V_{\ha} \longrightarrow U^- \times  \prod_{\gamma \in
    \Sc(\ha)} U_{\gamma},
$$ 
by $  j_1  :  V_{\ha} \longrightarrow U^-$\, and $j_2  : V_{\ha}
\longrightarrow  \prod_{\gamma \in \Sc(\ha)} U_{\gamma}$, respectively.

Observe that  $U_{-\ha} \simeq U_{-\ha} \hB / \hB = P_\ha / \hB \setminus (\hB / \hB)
$ is open in $ P_\ha/ \hB $. The image of $f_\ha$ is
$$
\Omega _{\ha}:= \{ \, [ (\widehat g : x )] \in \hG \times_{P_\ha} P_\ha / \hB \, : \,[ \widehat g ] \in V_{\ha} \,
\text{and} \,  s_{\ha} \inv \hy_{0} \inv j_2( [\widehat g]) )\inv j_1( [\widehat g ])) \inv g x
\in U_{-\ha} \hB / \hB \},
$$
where $[\widehat g]=\widehat g P_\halpha/P_\halpha\in\hG/P_\halpha$.
We deduce that $\Omega_{\ha}$ is open in $\hG \times_{\Pha} \Pha / \hB$.\\
Finally we prove that $f_\ha$ is an isomorphism. 
If we call $\phi_\ha: U_{-\ha} \hB / \hB \longto U_{-\ha}$ the inverse of the natural projection, then it's easy to see that the map
$$ \Omega_{\ha} \, \longto \, U^- \times \prod_{\gamma \in \Sc(\ha)} U_{\gamma} \times
               U_{-\ha}$$
that sends $  \bigl[( \widehat g, x) \bigr] $ \, to:
$$\bigl( (i_{\ha} \inv)_1( [ \widehat g]) \, \, , \, \,  (i_{\ha} \inv)_2( [\widehat g]) \,\, , \,\, \phi_{\ha} \bigl( s_{\ha} \inv \hy_{0} \inv j_2([\widehat g])\inv j_1([\widehat g[) \inv  gx \bigr) \bigr)$$
is the inverse of $f_\ha$.\\
To conclude, notice that for $u_{-\ha} = 1$, $f_{\ha}$ restrict to $i_{\ha}$, hence $\Omega_{\ha} \cap D_{\ha}$ is an open dense subset of the divisor. \\
For $u_{-\ha} \neq 1$, \, $\hy_0 s_{\ha} u_{-\ha} \hB / \hB  \neq \hy_0 s_{\ha}  \hB / \hB$ in $y_0 s_{\ha} \Pha/ \hB$. Hence $\hy_0 s_{\ha} u_{-\ha} \hB / \hB$ has to be a point of the open $G$-orbit, that is the orbit of $\hy_0  \hB / \hB$.\\
We conclude that $f_{\ha}$ maps \, $U^- \times \displaystyle \prod_{\gamma \in \Sc(\ha)} U_{\gamma} \times \bigl( U_{- \ha} \setminus \{1 \} \bigr) $ \, in the open $G$-orbit and that $ f_{\ha} \inv (D_{\ha})= U^-\times \prod_{\gamma \in \Sc(\ha)} U_{\gamma} \times\{1 \} 
$.
\end{proof}

\subsection{Conclusion}


In the  setting of the end of Section 3.1, we are now in position to characterize the image of the embedding: 

\begin{equation}
\label{embedd}
    H^0(G/B\times\hGB,\Li_\nu\otimes\Li_\hnu)^G \longto H^0(G/B,\Li_\nu)^{(H_0)_{\hy_0\hnu}}.
\end{equation}

Fix $\varphi \in (V_\nu^*) ^{(H_0)_{\hy_0\hnu}}$ and follow it throught the following isomorphisms: 
$$
\begin{array}{ccccl}
(V_\nu^*) ^{(H_0)_{\hy_0\hnu}}&\simeq&
H^0(G/B,\Li_\nu)^{(H_0)_{\hy_0\hnu}} &\simeq&
H^0(G/B\times G.\hy_0\hB/\hB,\Li_\nu\otimes\Li_\hnu)^G\\
\varphi&\mapsto&\sigma&\mapsto&\tilde\sigma.
\end{array}
$$
Fix $v_0\in V_\nu^{(B)}$ and $\tilde y_0\in(\Li_\hnu)_{\hy_0}-\{0\}$.
Explicitly, $\sigma$ and $\tilde \sigma$ are given by the formulas 
\begin{equation}
  \label{eq:1}
  \begin{array}{c}
\forall g\in G\qquad\sigma(gB/B)=[g:\varphi(gv_0)]\\[1em]
  \forall g_1,g_2\in G\qquad
  \tilde\sigma(g_1B/B,g_2\hy_0\hB/\hB)=[g_1:\varphi(g_2\inv g_1 v_0)]\otimes
  g_2\tilde y_0.
  \end{array}
\end{equation}

 We want to determine when $\ts$ extends to a global section. Take $\widehat\alpha\in \widehat\Delta$ that is a label of some edge descending
from the open orbit. Then $ D_{\ha} := \overline{G \hy_0
  s_\halpha\hB/\hB}$ is a divisor of $\hGB$ along which, we want to
determine the vanishing order of $\ts$.
Consider the image $V_{\ha}$ of the map $\iota_\ha$ defined by Proposition~\ref{charts}. 

\begin{prop}
\label{vanishneg}
\begin{enumerate}
\item Assume that  $\hy_0 \ha$ is negative. Then, the section $\ts$
  extends to a regular section on $G/B \times V_{\ha}$
  if and only if
  $$ X_{\beta}^m \varphi= 0 \, \, \, \text{for} \,\, \, m >  \scal{ - \hy_0 \cdot
  \hnu,(\hbeta^-)^\vee}.$$
\item Assume that  $\hy_0 \ha$ is positive. Then, the section $\ts$
  extends to a regular section on $G/B \times V_{\ha}$ if and only if
$$ X_{-\beta}^m \varphi= 0 \, \, \, \text{for} \,\, \, m > \scal{\hy_0 \cdot \hnu,(\hbeta^+)^\vee}.$$
\end{enumerate}
\end{prop}

\begin{proof}
Denote by $\begin{tikzcd}\epsilon_{\pm\beta}:= \,\CC \arrow[r] &
  \lg_{\pm\beta} \arrow[r, "\exp"] & U_{\pm\beta}\end{tikzcd}$ the
additive  one-parameter subgroups, 
and think about $\beta^\vee\,:\,\CC^*\longto T\subset G$ as a
one-parameter subgroup.
These three morphisms glue to give
a group homomorphism $\phi_\beta\,:\,\SL_2(\CC)\longto G$.  
The same notation is used, in the obvious way, also for $\hG$.
Now fix an $\sl_2$-triple $( X_{\ha}, H_{\ha}, X_{-\ha})$ for $\ha$, with $X_{\ha} \in \lg_{\ha}$. 

Suppose first that we are in the negative case, that is $\hy_0 \ha =-\hbeta^-$, where $\beta= - \rho (\hy_0 \ha) $.
Because of Proposition~\ref{charts}, $\ts$ extends to a section on $G/B \times V_{\ha}$ if and only if the map:

$$
\begin{array}{ccl}
  G/B \times  U^- \times \prod_{\gamma \in \Sc(\ha)} U_{\gamma} \times \CC^*
  &\longto&\Li_{\nu} \otimes \Li_{\hnu}\\
(gB/B, u^- , (u_{\gamma})_{\gamma} , t) &
\longmapsto&\ts  \bigl( (gB/B, u^-\prod_{i=1}^su_{\gamma_i} \hy_0s_{\ha}\epsilon_{-\ha}(t)\hB / \hB) \bigr)
\end{array}
 $$
extends at $t=0$. Notice also that  $U^-$ and $U_{\gamma}$ are subgroups of $G$ and $\ts$ is $G$-invariant, hence the function above extend at $t=0$ if and only if the following map, that with a little abuse will still be called $\ts$, extends at ${t=0}$.

\begin{equation}
\label{negcaseext}
\begin{array}{ccl}
  G/B \times \CC^* &\longto& \Li_{\nu} \otimes \Li_{\hnu}\\
 (gB/B , t) &\longmapsto & \ts \bigl( (gB/B \, , \, \hy_0 s_{\ha} \epsilon_{-\ha}(t) \hB / \hB) \bigr).
\end{array}
\end{equation}
Now fix $\sl_2$-triples for $\hbeta^\pm$ as in Lemma~\ref{sltriples},  so that for any $t \in \CC^*$,  $\epsilon_{ \pm \beta}(t)=\epsilon_{ \pm \widehat \beta^+}(t)\epsilon_{\pm \widehat \beta^-}(t)$.

In $\SL_2$, we have
$$
\begin{pmatrix}
  1&t\\0&1
\end{pmatrix}
\,
\begin{pmatrix}
  0&1\\-1&0
\end{pmatrix}=
\begin{pmatrix}
  1&0\\
t\inv&1
\end{pmatrix}
\,
\begin{pmatrix}
  -t&0\\0&-t\inv
\end{pmatrix}
\,
\begin{pmatrix}
  1&-t\inv\\
0&1
\end{pmatrix},
$$
for any $t \in \CC^*$. Hence 

 \begin{equation}
 \label{sl2formula}
     \ea(t)s_{\ha}=\ema(t \inv ) \ha^\vee(-t) \ea(-t \inv).
 \end{equation} 

Where $\ha$ can be replaced by any positive root of $G$ or $\hG$ for which a corresponding $\sl_2$-triple has been fixed. Now take $(gB/B , t) \in G/B \times \CC^*$, then

\begin{align*}
     ( gB/B \, , \, \hy_0 s_{\ha} \ema(t) \hB / \hB)   = &  ( gB/B \, , \, \hy_0  \ea(t) s_{\ha} \hB / \hB) \\
  = &  ( gB/B \, , \, \hy_0  \ema(t \inv) \hB / \hB)\\
  = &( gB/B \, , \, \epsilon_{ \widehat \beta^-}  ( c \inv t \inv) \hy_0 \hB / \hB )\\
\end{align*} 
where $c$ is the nonzero constant satisfying $ \hy_0 X_{\ha}= c X_{-\hbeta^-}$. 
But since $U_{\beta^+} \subset \hy_0 \hB \hy_0 \inv$:
$$  \epsilon_{ \widehat \beta^-}  ( c \inv t \inv) \hy_0 \hB / \hB  =  \epsilon_{ \widehat \beta^-}  ( c \inv t \inv)  \epsilon_{ \widehat \beta^+}  ( c \inv t \inv) \hy_0 \hB / \hB  =  \epsilon_{ \beta}  ( c \inv t \inv) \hy_0 \hB / \hB . $$
Now, Formula~\eqref{eq:1} gives
$$
\tilde\sigma\bigl(( gB/B \, , \, \hy_0 s_{\ha} \ema(t) \hB /
\hB)\bigr)=[g:\varphi(\eb ( -c \inv t \inv)gv_0)]\otimes \eb ( c \inv t
\inv)\tilde y_0.
$$
Moreover,
$$
\begin{array}{lll}
  \eb ( c \inv t
\inv)\tilde y_0&=\epsilon_{\hbeta^-}(c \inv t
\inv) \epsilon_{\hbeta^+}(c \inv t
\inv)\tilde y_0& \ \\
&=\epsilon_{-\hbeta^-}(ct)(s_{\hbeta^-})\inv (\hbeta^-)^\vee(-ct)\epsilon_{-\hbeta^-}(c\inv
  t\inv)\tilde y_0&{\rm\ by\ formula\ \eqref{sl2formula}}\\
&= \epsilon_{-\hbeta^-}(ct)(s_{\hbeta^-})\inv 
  (-ct)^{\scal{-\hy_0\hnu,(\hbeta^-)^{\vee}}}& {\rm\ since\ }U_{-\hbeta^-}\subset
                                    \hy_0\widehat U \hy_0\inv{\rm\ and\
                                    }\\
&&T{\rm\ acts\ with\ weight\ }
                                    -\hy_0\hnu.
\end{array}
$$
Rewrite the term $\varphi(\eb ( -c \inv t \inv)gv_0)$ as
$$
\begin{array}{ll}
  \varphi(\eb ( -c \inv t \inv)gv_0)&=\scal{\eb(c\inv t
                                     \inv)\varphi,gv_0}\\
&=\sum_{n\geq 0} \frac{(c\inv t
                                     \inv)^n}{n!}\scal{X_\beta^n\varphi,gv_0}.
\end{array}
$$
Finally, we get
$$
\tilde\sigma\bigl(( gB/B \, , \, \hy_0 s_{\ha} \ema(t) \hB /
\hB)\bigr)=[g: \sum_{n\geq 0} \pm\frac{(c\inv t
                                     \inv)^{n +\scal{\hy_0\hnu,(\hbeta^-)^{\vee}}}}{n!}\scal{X_\beta^n\varphi,gv_0}]\otimes
                                   (\epsilon_{-\hbeta^-}(ct)(s_{\hbeta^-})\inv\tilde y_0).
$$
The term $\epsilon_{-\hbeta^-}(ct)(s_{\hbeta^-})\inv \tilde y_0$ is regular
on $\CC$.
Hence $\tilde\sigma$ has no pole along $t=0$ if and only if 
$$
\forall n > \scal{ -\hy_0\hnu , (\hb^-)^{\vee} }\Longrightarrow 
\scal{X_\beta^n\varphi,gv_0}=0\in\CC[G].
$$ 
But $Gv_0$ spans $V_\nu$ that is irreducible. As a consequence,
$\scal{X_\beta^n\varphi,gv_0}=0$ if and only if $X_\beta^n\varphi=0$.

\bigskip
Suppose now that we are in the positive case, hence $\hy_0 \cdot \ha= \hbeta^+$ with $\beta= \rho ( \hy_0 \ha)$. The outline of the proof doesn't change. By the same argument of the previous case, $\ts$ extends to a section on $G/B \times V_{\ha}$ if and only if the following map, that will still be called $\ts$, extends at $t=0$.
\begin{equation}
\label{poscaseext}
\begin{array}{ccl}
  G/B \times \CC^* &\longto& \Li_{\nu} \otimes \Li_{\hnu}\\
 (gB/B , t) &\longmapsto & \ts \bigl( (gB/B \, , \, s_{\beta} \hy_0 s_{\ha} \epsilon_{-\ha}(t) \hB / \hB) \bigr)
\end{array}
\end{equation}
Again fix $\sl_2$-triples for $\hbeta^{\pm}$ as in Lemma~\ref{sltriples}. Take $(gB/B , t) \in G/B \times \CC^*$, then

\begin{align*}
     ( gB/B \, , \, s_{\beta} \hy_0 s_{\ha} \ema(t) \hB / \hB)   = & ( gB/B \, , \, s_{\beta}  \hy_0  \ema(t \inv) \hB / \hB)  \\
  = &( gB/B \, , \,  s_{\beta} \epsilon_{- \beta}  ( c \inv t \inv) \hy_0 \hB / \hB )\\
\end{align*} 
where now $c$ is the nonzero constant satisfying $ \hy_0 X_{\ha}= c X_{\hbeta^+}.$ Hence:
$$
\tilde\sigma\bigl(( gB/B \, , \, s_{\beta} \hy_0 s_{\ha} \ema(t) \hB /
\hB)\bigr)=[g:\varphi(\emb ( -c \inv t \inv)(s_{\beta})\inv gv_0)]\otimes s_{\beta}\emb ( c \inv t
\inv)\tilde y_0.
$$
Similarly to the previous computations:
$$
\begin{array}{ll}
  \emb ( c \inv t
\inv)\tilde y_0&=\epsilon_{-\hbeta^+}(c \inv t
\inv) \epsilon_{-\hbeta^-}(c \inv t
\inv)\tilde y_0 \\
&=\epsilon_{\hbeta^+}(ct)s_{\hbeta^+} \epsilon_{\hbeta^+}(c\inv
  t\inv) (\hbeta^+)^\vee(-c\inv t \inv )\tilde y_0\\
&= \epsilon_{\hbeta^+}(ct)s_{\hbeta^+} 
  (-c \inv t \inv )^{\scal{-\hy_0\hnu,(\hbeta^+)^{\vee}}}.
\end{array}
$$
While for the other term

$$
\begin{array}{ll}
  \varphi(\emb ( -c \inv t \inv) s_{\beta} \inv gv_0)&=\scal{\emb(c\inv t
                                     \inv)\varphi, s_{\beta} \inv gv_0}\\
&=\sum_{n\geq 0} \frac{(c\inv t
                                     \inv)^n}{n!}\scal{X_{-\beta}^n\varphi, s_{\beta} \inv gv_0}.
\end{array}
$$
Finally we get
$$
\tilde\sigma\bigl(( gB/B \, , \, s_{\beta} \hy_0 s_{\ha} \ema(t) \hB /
\hB)\bigr)=[g: \sum_{n\geq 0} \pm\frac{(c\inv t
                                     \inv)^{n - \scal{\hy_0\hnu,(\hbeta^+)^{\vee}}}}{n!}\scal{X_{-\beta}^n\varphi, s_{\beta} \inv gv_0}]\otimes
                                   (s_{\beta} \epsilon_{\hbeta^+}(ct)s_{\hbeta^+})\tilde y_0).
$$
By the same argument of the previous case we conclude that $\tilde\sigma$ has no pole along $t=0$ if and only if 
$$
\forall \, n > \scal{ \hy_0\hnu , (\hb^+)^{\vee} }\Longrightarrow 
X_\beta^n\varphi =0.
$$ 
\end{proof}

\begin{remark}
 If $\hy_0 \ha = - \beta^-$, then $\hy_0 \cdot \ha^\vee= -(\widehat \beta^-)^\vee$, hence:
     $$  \scal{ - \hy_0 \cdot \hnu, (\widehat \beta^-)^\vee} = \scal{\hy_0 \cdot \hnu , \hy_0 \cdot \ha^\vee } = \scal{ \hnu , \ha^\vee}. $$
 Similarly, if $\hy_0 \ha = \widehat \beta^+$, then $\hy_0 \cdot \ha^\vee= (\widehat \beta^+)^\vee$, which implies:
  $$  \scal{  \hy_0 \cdot \hnu, (\widehat \beta^+)^\vee} = \scal{\hy_0 \cdot \hnu , \hy_0 \cdot \ha^\vee } = \scal{ \hnu , \ha^\vee}. $$

\end{remark}
Recall that in the introduction we set $\mathcal{D}= \{ \ha  \in \hPhi
\, : \, \rho(\hy_0\ha) \not \in \Phi^1 \}$, for $\ha \in \mathcal{D}$,
let $V_{\ha}$ be as in Proposition~\ref{vanishneg}. We prove the theorem of the introduction.

\begin{proof}[Proof of Theorem 1]
The first condition of the theorem is implied by
Lemma~\ref{H^0-seminvariants}. Then set $V:=  \bigcup_{\ha \in
  \mathcal{D}} G/B \times V_{\ha} $, which is open in $G/B \times
\hGB$. The complement of $V$ is of codimension strictly larger then 1
and $G/B\times\hG/\hB$ is normal, hence the restriction map
$$H^0(G/B \times \hGB, \Li_{\nu} \otimes \Li_{\hnu}) \longto H^0(V, \Li_{\nu} \otimes \Li_{\hnu}) $$
is an isomorphism. Then, from Proposition~\ref{vanishneg} and from the previous remark we deduce that $\varphi \in (V_{\nu}^*)^{(H^0)_{\hy_0\hnu}}$ is in the image of $$ H^0(G/B \times \hGB, \Li_{\nu}\otimes\Li_{\hnu}) \longto (V_{\nu}^*)^{(H^0)_{\hy_0\hnu}} $$
if and only if the second condition of the theorem holds.
\end{proof}

As remarked in the introduction, the conditions of
Theorem~\ref{th:mainintro} are in general redundant. This is linked to
the fact in $\Gamma( \hG / G)$ we may have different edges between the
same vertices. It seems not easy to determine a minimal set of
conditions in a uniform way. The following lemma checks that, for the
tensor product case, a minimal set of conditions is described by the
set: 
$$ \{ ( \alpha , 0) \in \Phi \times \Phi \, : \, \alpha \in \Delta \}. $$

\begin{lemma}
\label{samecondition}
 For $\beta \in \Phi^+$, and $\sigma \in V^*_{\nu}(\hy_0 \cdot \hnu)$, the following are equivalent: 
 \begin{enumerate}
     \item[1)] $X_{\beta} ^ m \cdot \sigma = 0 $ \, \, for \, \, $m \, > \, \scal{ - \hy_0 \cdot
  \hnu,(\hbeta^-)^\vee} .$
     \item[2)] $X_{-\beta} ^ m \cdot \sigma = 0 $ \, \, for \, \, $m \, > \, \scal{ \hy_0 \cdot
  \hnu,(\hbeta^+)^\vee}.$
 \end{enumerate}
\end{lemma}

\begin{remark}
If $\beta \in \Phi ^+$ and $\ha_1 \, , \ha_2 \in \widehat \Delta$ satisfy: $$\hy_0 \cdot \ha_1 = - \hbeta^- \, \, \text{and} \, \, \, \hy_0 \cdot \ha_1 =  \hbeta^+$$
then $\hy_0  s_{\ha_1}= s_{\beta} \hy_0  s_{\ha_2}  $, hence their
$G$-orbit is the same, and by the proof we have done we expect that
$\ha_1$ and $\ha_2$ give the same condition in
Theorem~\ref{th:mainintro}. This is directly checked in previous lemma.
\end{remark}

\begin{proof}
Call $\sl_{\beta} $ the subalgebra of $\lg$ spanned by $ X_{\beta},
H_{\beta}, X_{-\beta}$. Decompose $V^*_{\nu}$ into a direct sum of
$\sl_{\beta} $ irreducible representations: $ V^*_{\nu} = \bigoplus
V_{\delta} $.
Write $\sigma = \sum_\delta \sigma_{\delta}$ accordingly to this
decomposition. 
Observe that $\sigma_{\delta} \in V_{\delta}(\hy_0 \cdot \hnu_{| \CC H_{\beta}}).$
Since $V_{\delta}$ is an $\sl_2$ irreducible representation, if
$\sigma_{\delta} \neq 0$, then: 
$$
\begin{array}{ll}
X_{\beta}^m \cdot \sigma_{\delta} = 0 \, \iff \, \scal{  \hy_0 \cdot
  \hnu,\beta^\vee} + 2m > \scal{ \delta ,\beta^\vee} $$,&{\rm and}\\[0.8em]
X_{-\beta}^m \cdot \sigma_{\delta} = 0 \, \iff \,  \scal{  \hy_0 \cdot \hnu,\beta^\vee} - 2m < - \scal{ \delta ,\beta^\vee}.\\
\end{array}
$$
Hence condition 1) is equivalent to: 
$$  \forall \, \delta: \sigma_\delta \neq 0, \, \Longrightarrow \,  \scal{  \hy_0 \cdot \hnu,\beta^\vee} -2 \scal{\hy_0 \cdot \hnu , (\widehat \beta ^-)^\vee } +2 > \scal{ \delta, \beta^\vee}. $$
Similarly condition 2) is equivalent to:
$$  \forall \, \delta: \sigma_\delta \neq 0, \, \Longrightarrow \,   \scal{  \hy_0 \cdot \hnu,\beta^\vee} -2 \scal{\hy_0 \cdot \hnu , (\widehat \beta ^+)^\vee } -2 < - \scal{ \delta, \beta^\vee}. $$
But Lemma~\ref{sltriples} implies that $ \scal{  \hy_0 \cdot \hnu,\beta^\vee}= \scal{  \hy_0 \cdot \hnu,(\widehat \beta^-)^\vee}+\scal{  \hy_0 \cdot \hnu,(\widehat \beta^+)^\vee} $, then we easily conclude that 1) and 2) are equivalent.
\end{proof}

\section{Explicit description on the examples}
\label{sec:expl}

For each example, we determine a working $\hy_0$ and a set of simple
roots parametrizing the $G$-stable divisors of $\hGB$.
With the notation of the introduction, we give a subset $\Dc_0$ of
$\Dc$ such that the map $\ha\longmapsto \overline{G\hy_0s_\ha \hB / \hB}$
is a bijection from $\Dc_0$ onto the set of $G$-stable divisors.

\subsection{Tensor product case}

Here $\hG=G\times G$, $\hT=T\times T$ and $\hB=B\times B$. 
Moreover $\Phi^1$ is empty,  $X( \widehat T)= X(T) \times X(T)$ and $\rho( \lambda, \mu)=\lambda+ \nu$.
Set $\hy_0=(e,w_0)\in\hW=W\times W$.
It is clear that $G\hy_0$ is open in $\hGB$ and that
$\ell(\hy_0)=\dim(\hGB)-\dim(G/B)$.
By the Bruhat decomposition $\{G(B/B,w_0s_\alpha
B/B)\,:\,\alpha\in\Delta\}$ is the collection of codimension one
$G$-orbits in $\hG/\hB$. 
In particular the set 
 $$
\Dc_0^-=\{(0,\alpha)\,:\,\alpha\in\Delta\}
$$ 
works. Another possible choice is $$ \Dc_0^+=\{(\alpha, 0)\,:\,\alpha\in\Delta\}. $$
Note that $\hy_0=\hy_0 \inv $ and that, for $\alpha \in \Delta$,\,  $\hy_0  (0, \alpha)=( 0 , w_0 \alpha) \in \hPhi^-$, while $\hy_0 ( \alpha, 0)=( \alpha, 0 ) \in \hPhi^+$. So, according to our notations, $(0, \alpha)= \ha^-$ and $(\alpha, 0)=\ha^+$.
If $ \nu \in X(T)^+$ and $\lambda , \mu \in X(T)$ we define two sets:
$$
 V^+( \nu, \lambda , \mu):= \{ v \in V_\nu(\lambda) \, : \, \forall \, \alpha \in \Delta \,, \,  X_\alpha^m v= 0 \, \, \text{for} \, \,  m > \scal{ \mu , \alpha^\vee} \}  $$
$$ 
V^-( \nu, \lambda , \mu):= \{ v \in V_\nu(\lambda) \, : \, \forall \, \alpha \in \Delta \,, \,  X_{-\alpha}^m v= 0 \, \, \text{for} \, \,  m > \scal{ \mu , \alpha^\vee} \}  $$

From now on, fix $\nu, \nu_1 , \nu_2 \in X(T)^+$ and set $\hnu=(\nu_1,
\nu_2) \in X(T \times T)^+$. We denote by $\nu^*:= -w_0\nu$, so that $V_{\nu^*} \simeq V_\nu^*$.
In Theorem 2.1 of \cite{PRV} the autors realized, by algebraic methods isomorphisms:
\begin{equation}
    \label{PRVminus}
\Mult (\nu, \hnu^*) \simeq V^-( \nu , \nu_1 - \nu_2^* , \nu_1)
 \end{equation}
 \begin{equation}
     \label{PRVplus}
 \Mult (\nu, \hnu^*) \simeq V^+( \nu , \nu_2 - \nu_1^* , \nu_1^*)
\end{equation}

We explain how this can be recovered from Theorem~\ref{th:mainintro}, using $\Dc_0^+$ and $\Dc_0^-$ as parametrizations of the $G$-stable divisors.
Notice that,  in this case, the conditions 1 of the theorem are empty. Then $\rho (\hy_0 \hnu)= (\nu_1 - \nu_2^*)$ and for $\halpha^+=(\alpha, 0) \in \Dc_0^+$, $\rho(\hy_0 \ha^+)= \alpha $ and $\scal{\hnu, (\ha)^\vee}= \scal{ \nu_1 , \alpha^\vee}$. Hence, from Theorem~\ref{th:mainintro}, we deduce that $$ \Mult( \nu^*, \hnu) \simeq V^-( \nu, \nu_1 - \nu_2^*, \nu_1).$$
Since $\Mult(\nu, \hnu^*)$ is  isomorphic to $\Mult( \nu^*, \hnu)$, we recover \eqref{PRVminus}.\\
Now, $\rho(\hy_0 (0, -w_0 \alpha))= -\alpha$, and $ \scal{\hnu, (0, -w_0 \alpha)^\vee}= \scal{ \nu_2, (-w_0 \alpha)^\vee}= \scal{ \nu_2^*, \alpha^\vee}$. Then, if we use $\Dc_0^-$ to get a minimal number of conditions in the second point of the theorem, we deduce that: $$ \Mult (\nu^*, \hnu) \simeq V^+(\nu, \nu_1-\nu_2^*, \nu_2^*).$$
And since $\Mult(\nu, \hnu^*) \simeq
\Mult (\nu^*, \hnu) \simeq \Mult(\nu^*, (\nu_2, \nu_1))$, we recover
\eqref{PRVplus}.

\subsection{$\Sp_{2n}$ in $\SL_{2n}$}

Fix $n\geq 2$.
Let $V$ be a $2n$-dimensional vector space 
with fixed basis $\base=(e_1,\dots,e_{2n})$.
Consider the following matrices
\begin{eqnarray}
  \label{eq:defJn}
J_n=\left(
  \begin{array}{ccc}
    &&1\\
&\revddots\\
1
  \end{array}
\right);
\qquad{\rm and}\qquad
\omega=\left(
  \begin{array}{cc}
   0 &J_n\\
-J_n&0
  \end{array}
\right).\label{eq:defsympform}
\end{eqnarray}
of size $n\times n$ and $2n\times 2n$ respectively.
View $\omega$ as a symplectic bilinear form of $V$.  

Let $G$ be the associated symplectic group.
Set $T=\{\diag(t_1,\dots,t_{n},t_{n}^{-1},\dots,t_1^{-1})\,:\,t_i\in\CC^*\}$.
Let $B$ be the Borel subgroup of $G$ consisting of upper triangular matrices of $G$.
For $i\in [1,n]$, let $\varepsilon_i$ denote the character of $T$ that maps 
$\diag(t_1,\dots,t_{n},t_{n}^{-1},\dots,t_1^{-1})$ to $t_i$; then 
$X(T)=\oplus_i\ZZ\varepsilon_i$.
Here
$$
\begin{array}{l}
  \Phi^+=\{\varepsilon_i\pm\varepsilon_j\,:\,1\leq i<j\leq n\}\cup 
\{2\varepsilon_i\,:\,1\leq i\leq n\}, \\
\Delta=\{\alpha_1=\varepsilon_1-\varepsilon_2,\,\alpha_2=\varepsilon_2-\varepsilon_3,\dots,\,
\alpha_{n-1}=\varepsilon_{n-1}-\varepsilon_{n},\,\alpha_{n}=2\varepsilon_{n}\},
  {\rm\ and}\\
X(T)^+=\{\sum_{i=1}^{n}\lambda_i \varepsilon_i\,:\,
  \lambda_1\geq\cdots\geq \lambda_{n}\geq 0\}.
\end{array}
$$
For $i\in [1;2n]$, set $\overline{i}=2n+1-i$.
The Weyl group $W$ of $G$ is a subgroup of the Weyl group $S_{2n}$ of
$\hG=\SL(V)$. More precisely
$$
W=\{w\in S_{2n}\,:\,w(\overline{i})=\overline{w(i)} \ \ \forall i\in [1;2n]\}.
$$

Note that $
\dim G=\dim \GL_{2n}(\CC)-\dim\wedge^2V^*=2n^2+n
$ and
$
\dim \hGB=n(2n-1).
$
Hence we are looking for $\hy_0\in \hW$ such that
$$
\dim G_{\hy_0\hB/\hB}=2n.
$$
This is consistent with the fact that $\sharp\Phi_\lh^1=n$
($\Phi_\lh^1=\{2\varepsilon_i\,:\,1\leq i\leq n\}$).

An element $\hw\in \hW$ is written as a word $[\hw(1)\,\hw(2)\,\dots
\hw(2n)]$. Set 
$$
\hy_0=[1\,\bar 1\,2\,\bar 2\dots].
$$
Since the $\omega$-orthogonal of $\scal{e_1,e_{\bar 1},\dots,e_{\bar
  k}}$ is  $\scal{e_{k+1},\dots,e_{\bar
  n}}$ the stabilizer of $\hy_0\hB/\hB$ is diagonal by blocks with 2 blocks
of size $2k$ and $2(n-k)$ (in the basis ordered according to $\hy_0$).
Since this is true for any $k$, this stabilizer consists in block
diagonal matrices with blocks of size 2.
Moreover, in each block the matrix has to be triangular because of its
action on $e_1,\dots,e_n$.
Moreover, these blocks belong to $\Sp(2)=\SL(2)$. We just proved that
$G_{\hy_0\hB/\hB}$ is contained in a group of dimension $2n$. 
For dimension reasons, we deduce that $G_{\hy_0\hB/\hB}$ is equal to
this subgroup and that $G. \hy_0\hB/\hB$ in open in $\hGB$.\\

The length of $\hy_0$ is the number of inversions, that is the set of
pairs $(i<j)$ such that $j$ occurs before $i$ in the word.
Fix $i\in\{1,\dots,n\}$. The number of $j>i$ such that $j$ occurs
before $i$ in the word $\hy_0$ is $i-1$.
The contribution to these pairs to $\ell(\hy_0)$ is
$
\frac{n(n-1)}2.
$
Similarly, the contribution of the pairs $\bar j>\bar i$ with $\bar i\in\{n+1,\dots,2n\}$ is
$
\frac{n(n-1)}2.
$
Hence 
$
\ell(\hy_0)=n(n-1)=\dim \hGB-\dim(G/B)
$
and
$y_0$ has minimal length.\\

Let $\hT$ be the maximal torus of $\hG$ consisting in diagonal
matrices. Write $X(\hT)=\oplus_{i=1}^{2n}\ZZ
\hat\varepsilon_i/(\hat\varepsilon_1+\cdots+\hat\varepsilon_{2n})$
with the usual notation.
Let $\halpha_i=\hat\varepsilon_i-\hat\varepsilon_{i+1}$ be a simple root of
$\hG$. 
If $i=2k$ is even then $\hy_0 \halpha_i=\hat\varepsilon_{\bar
  k}-\hat\varepsilon_{k+1}\in\hPhi^1$.
If $i=2k+1$ is odd then $\hy_0 \halpha_i=\hat\varepsilon_{
  k+1}-\hat\varepsilon_{\overline{k+1}}\in\hPhi^2$.
By Lemma~\ref{lem:carddiv} 
$$
\Dc_0=\{\halpha_{2k}\,:\, k=1,\dots,n-1\}
$$
works. Moreover 
$$
-\rho(\halpha_{2k})=\varepsilon_k+\varepsilon_{k+1}\qquad \forall k=1,\dots,n-1.
$$

\subsection{$\Spin_{2n-1}$ in $\Spin_{2n}$}

Let $V$ be a $2n$-dimensional vector space endowed with a basis 
$\base=(e_1,\dots,e_{2n})$.
Denote by $(x_1,\dots,x_{2n})$ the dual basis. 
For $i\in [1;2n]$, set $\overline{i}=2n+1-i$.
Let $\hG_0$ be the orthogonal  group associated to the quadratic form 
$$
Q=\sum_{i=1}^nx_ix_{\overline{i}}.
$$ 

Set $\hT_0=\{{\diag}(t_1,\dots,t_n,t_n^{-1},\dots,t_1^{-1})\,:\,t_i\in\CC^*\}$ in
$\hG_0$.
Let $\hB_0$ be the Borel subgroup of $\hG_0$ consisting in upper triangular matrices of $\hG_0$.
Let $\hat\varepsilon_i$ denote the character of $\hT_0$ that maps 
$\diag(t_1,\dots,t_n,t_n^{-1},\dots,t_1^{-1})$ to $t_i$; then 
$X(\hT_0)=\sum_{i=1}^n\ZZ\hat\varepsilon_i$.
Here 
$$
\begin{array}{l}
  \hPhi^+=\{\hat\varepsilon_i\pm\hat\varepsilon_j\,:\,1\leq i<j\leq n\}, {\rm\ and}\\
\widehat\Delta=\{\halpha_1=\hat\varepsilon_1-\hat\varepsilon_2,\,\halpha_2=\hat\varepsilon_2-\hat\varepsilon_3,\dots,\,
\halpha_{n-1}=\hat\varepsilon_{n-1}-\hat\varepsilon_n,\,\halpha_n=\hat\varepsilon_{n-1}+\hat\varepsilon_n\}.
\end{array}
$$

The Weyl group $\hW$ of $\hG_0$ is a subgroup of the Weyl group $S_{2n}$
of $\SL(V)$. More precisely
$$
\hW=\{w\in S_{2n}\,:\,
\left\{\begin{array}{l}
w(\overline{i})=\overline{w(i)} \ \ \forall i\in [1;2n]\\
\sharp w([1;n])\cap[n+1;2n])\mbox{ is even}
\end{array}
\right .
\}.
$$

For $i=1,\dots,n-1$, $s_{\widehat\alpha_i}=(i,i+1)(\overline{i+1}\;\bar
i)$. Moreover $s_{\widehat\alpha_n}=(n-1\,\bar n) (n\, \overline{n-1})$.

\bigskip
Let $\Hc=\scal{e_1,\dots,e_{n-1},\frac{e_n+e_{n+1}}{\sqrt
    2},e_{n+1},\dots,e_{2n}}$ with coordinates $(x_i)_{1\leq i\leq
  n-1}\cup (y)\cup (x_{\bar i})_{1\leq i\leq
  n-1}$. The restriction of $Q$ on $\Hc$ is
$$
Q_{|\Hc}=y^2+\sum_{i=1}^{n-1}x_ix_{\overline{i}}.
$$
The stabilizer of $\Hc$ in $\hG_0$ is $G_0=\SO_{2n-1}(\CC)$.  
Its maximal torus $T_0$ is
$$
\{\diag(t_1,\dots,t_{n-1},1,1,t_{n-1}^{-1},\dots,t_1^{-1})\,:\,t_i\in\CC^*\}.
$$ 
The groups $G$ and $\hG$ are the universal covers of $G_0$ and $\hG_0$ respectively.
Here
$$ 
\begin{array}{l}
  \Phi^+=\{\varepsilon_i\pm\varepsilon_j\,:\,1\leq i<j\leq n-1\}\cup 
\{\varepsilon_i\,:\,1\leq i\leq n-1\}, {\rm\ and}\\
\Delta=\{\alpha_1=\varepsilon_1-\varepsilon_2,\,\alpha_2=\varepsilon_2-\varepsilon_3,\dots,\,
\alpha_{n-2}=\varepsilon_{n-2}-\varepsilon_{n-1},\,\alpha_{n-1}=\varepsilon_{n-1}\}.
\end{array}
$$

Note that $\Delta^2=\{\alpha_{n-1}\}$ and there is only one $G_0$-stable
divisor in $\hG/\hB$.
We have
$$
\dim\hGB=n(n-1)\qquad \dim(G/B)=(n-1)^2.
$$

Set
$$
\hy_0=s_{\widehat\alpha_{n-1}}\dots s_{\widehat\alpha_{2}}
s_{\widehat\alpha_{1}}=(1\,n\,n-1\,\dots 2) (\bar 1\,\bar n\,\dots \bar 2).
$$
We have $\hy_0\widehat\epsilon_i=\widehat\epsilon_{i-1}$ for $i=2,\dots,n$ and
$\hy_0\widehat\epsilon_1=\widehat\epsilon_n$.
In particular 
$$
\hy_0\hPhi^+=\{\widehat\epsilon_n\pm \widehat\epsilon_j\,:\,
j=1,\dots,n-1\}\cup \{\widehat\epsilon_i\pm \widehat\epsilon_j\,:\, 1\leq i<j=n-1\}.
$$
One easily deduce that  the stabilizer of $\hy_0$ in $G_0$ is $H_0$.

The only descent is $\widehat\alpha_{1}$. Then
$\Dc_0=\Dc=\{\widehat\alpha_1\}$
and $-\rho(\hy_0\widehat\alpha_1)=-\rho(\hat\varepsilon_n-\hat\varepsilon_1)=\varepsilon_1$.

\subsection{$G_2$ in $\Spin_7$}

Here $G$ is the group of type $G_2$ embedded in $\Spin_7(\CC)=\hG$
using the first fundamental representation which is 7-dimensional. 
We label the simple roots as follows:

\begin{center}
  \begin{tikzpicture}[scale=0.5]
    \draw (-1,0) node[anchor=east] {$B_{3}$}; \draw (0 cm,0) -- (2
    cm,0); \draw (2 cm, 0.1 cm) -- +(2 cm,0); \draw (2 cm, -0.1 cm) --
    +(2 cm,0); \draw[shift={(3.2, 0)}, rotate=0] (135 : 0.45cm) --
    (0,0) -- (-135 : 0.45cm); \draw[fill=white] (0 cm, 0 cm) circle
    (.25cm) node[below=4pt]{$1$}; \draw[fill=white] (2 cm, 0 cm)
    circle (.25cm) node[below=4pt]{$2$}; \draw[fill=white] (4 cm, 0
    cm) circle (.25cm) node[below=4pt]{$3$};
  \end{tikzpicture}\hspace{2cm}
  \begin{tikzpicture}[scale=0.5]
    \draw (-1,0) node[anchor=east] {$G_2$}; \draw (0,0) -- (2 cm,0);
    \draw (0, 0.15 cm) -- +(2 cm,0); \draw (0, -0.15 cm) -- +(2 cm,0);
    \draw[shift={(0.8, 0)}, rotate=180] (135 : 0.45cm) -- (0,0) --
    (-135 : 0.45cm); \draw[fill=white] (0 cm, 0 cm) circle (.25cm)
    node[below=4pt]{$1$}; \draw[fill=white] (2 cm, 0 cm) circle
    (.25cm) node[below=4pt]{$2$};
  \end{tikzpicture}
\end{center}

The map $\rho$ is characterized by 
$$
\rho(\halpha_2)=\alpha_2\quad
\rho(\halpha_1)=\rho(\halpha_3)=\alpha_1,
$$
and satisfies
$$
\begin{array}{c}
\rho(\halpha_1+\halpha_2)=\rho(\halpha_2+\halpha_3)=\alpha_1+\alpha_2\\[0.8em]
\rho(\halpha_2+2\halpha_3)=\rho(\halpha_1+\halpha_2+\halpha_3)=2\alpha_1+\alpha_2
\end{array}
$$
In particular
$$
\widehat\Phi^+\cap\hPhi^1=\{\halpha_2, \halpha_1+\halpha_2+2\halpha_3, \halpha_1+2\halpha_2+2\halpha_3\}.
$$

The working $\hy_0$ are
$$
\hy_0=\hs_1\, \hs_2\, \hs_3\quad{\rm and}\quad \hy_0=\hs_3\, \hs_2\, \hs_3
$$
Indeed, one can check that $\ell(\hy_0)=3$ and $\dim(G\cap
\hy_0\hB\hy_0\inv)=3$.

Moreover, the only simple roots $\widehat\alpha$ such that $\dim(G\cap
\hw \hB\hw\inv)=4$ where $\hw=\hy_0s_\halpha$ is $\halpha_3$. 
The so obtained divisor is the only $G$-stable divisor accordingly to $\sharp\Delta^2=1$.

There are $\frac{\sharp\hW}{\sharp W}=4$ $G$-orbits in $\hGB$.
The closed orbit has dimension 6 and the open one 9: hence there is
one orbit in each dimension from 6 to 9. Computing the dimensions of
the stabilizers we get the graph $\Gamma(\hG/G)$:

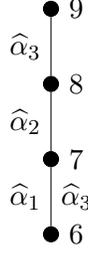
\begin{center}
    \begin{tikzpicture}[scale=1,el/.style = {inner sep=2pt, align=left, sloped}]
      \node[dnode,fill=black] (a) [label=right:{$6$}] at (0,0) {}; 
      \node[dnode,fill=black] (b) [label=right:{$7$}] at (0,1) {}; 
      \node[dnode,fill=black] (c) [label=right:{$8$}] at (0,2) {}; 
      \node[dnode,fill=black] (d) [label=right:{$9$}] at (0,3) {};
       \draw (a) -- node[left] {$\halpha_1$}node[right] {$\halpha_3$} (b);
\draw (b) -- node[left] {$\halpha_2$} (c);
\draw (c) -- node[left] {$\halpha_3$} (d);
    \end{tikzpicture}
  \end{center}

Here
$$\Dc_0=\Dc=\{\widehat\alpha_3\}
\quad{\rm and}\quad
 -\rho(\hy_0\widehat\alpha_3)=\alpha_1+\alpha_2.
$$

\subsection{$F_4$ in $E_6$}

Here $G$ is the group of type $F_4$ embedded in $\hG$ of type $E_6$.
We label the simple roots as follows:

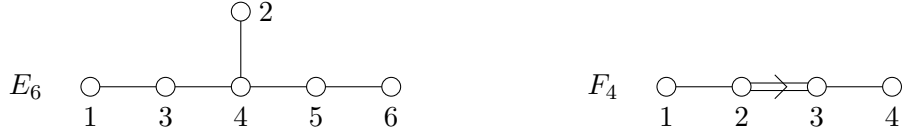
\begin{center}
  \begin{tikzpicture}[scale=0.5]
    \draw (-1,0) node[anchor=east] {$E_6$}; \draw (0 cm,0) -- (8
    cm,0); \draw (4 cm, 0 cm) -- +(0,2 cm); \draw[fill=white] (0 cm, 0
    cm) circle (.25cm) node[below=4pt]{$1$}; \draw[fill=white] (2 cm,
    0 cm) circle (.25cm) node[below=4pt]{$3$}; \draw[fill=white] (4
    cm, 0 cm) circle (.25cm) node[below=4pt]{$4$}; \draw[fill=white]
    (6 cm, 0 cm) circle (.25cm) node[below=4pt]{$5$};
    \draw[fill=white] (8 cm, 0 cm) circle (.25cm)
    node[below=4pt]{$6$}; \draw[fill=white] (4 cm, 2 cm) circle
    (.25cm) node[right=3pt]{$2$};
  \end{tikzpicture}
\hspace{2cm}
\begin{tikzpicture}[scale=0.5]
  \draw (-1,0) node[anchor=east] {$F_4$}; \draw (0 cm,0) -- (2 cm,0);
  \draw (2 cm, 0.1 cm) -- +(2 cm,0); \draw (2 cm, -0.1 cm) -- +(2
  cm,0); \draw (4.0 cm,0) -- +(2 cm,0); \draw[shift={(3.2, 0)},
  rotate=0] (135 : 0.45cm) -- (0,0) -- (-135 : 0.45cm);
  \draw[fill=white] (0 cm, 0 cm) circle (.25cm) node[below=4pt]{$1$};
  \draw[fill=white] (2 cm, 0 cm) circle (.25cm) node[below=4pt]{$2$};
  \draw[fill=white] (4 cm, 0 cm) circle (.25cm) node[below=4pt]{$3$};
  \draw[fill=white] (6 cm, 0 cm) circle (.25cm) node[below=4pt]{$4$};
\end{tikzpicture}
\end{center}
The root system $E_6$ lies in $\RR^8$ with basis $(\widehat\epsilon_i)_{1\leq
  i\leq 8}$. More precisely it spans  the  space $x_6=x_7=-x_8$. See \cite{Bou}.
The roots are 
$$\pm \widehat\epsilon_i\pm\widehat\epsilon_j \quad 1\leq i<j\leq 5
\quad{\rm and}\quad
\pm\frac 1 2 (\widehat\epsilon_8-\widehat\epsilon_7-\widehat\epsilon_6+\sum_{i=1}^5
(-1)^{\nu(i)}\widehat\epsilon_i\qquad\sum\nu(i) {\rm\ even.}
$$
Set
$\tilde\epsilon_6=\widehat\epsilon_8-\widehat\epsilon_6-\widehat\epsilon_7$.
Then,
the simple roots are
$$
\halpha_1=\frac 1 2 \bigg
(\widehat\epsilon_1+\widehat\epsilon_8-(\widehat\epsilon_2+\widehat\epsilon_3+\widehat\epsilon_4+\widehat\epsilon_5+\widehat\epsilon_6+\widehat\epsilon_7)\bigg) =\frac 1 2 \bigg
(\widehat\epsilon_1+\tilde\epsilon_8-(\widehat\epsilon_2+\widehat\epsilon_3+\widehat\epsilon_4+\widehat\epsilon_5)\bigg)
$$
and 
$$
\halpha_2=\widehat\epsilon_1+\widehat\epsilon_2\quad
\halpha_3=\widehat\epsilon_2-\widehat\epsilon_1\quad
\halpha_4=\widehat\epsilon_3-\widehat\epsilon_2\quad
\halpha_5=\widehat\epsilon_4-\widehat\epsilon_3\quad
\halpha_6=\widehat\epsilon_5-\widehat\epsilon_4.
$$

\bigskip
{\bf Root system $F_4$.}
The roots are 
$$\frac 1 2 (\pm\epsilon_1\pm \epsilon_2\pm \epsilon_3\pm \epsilon_4)
\quad{\rm and}\quad
\pm\epsilon_i\qquad
\pm \epsilon_i\pm\epsilon_j \quad 1\leq i<j\leq 4.
$$
The simple roots are
$$\alpha_1=\epsilon_2-\epsilon_3\quad
\alpha_2=\epsilon_3-\epsilon_4\quad
\alpha_3=\epsilon_4\quad
\alpha_4=\frac 1 2 (\epsilon_1- \epsilon_2- \epsilon_3-
\epsilon_4).
$$
The map $\rho$ is characterized by 
$$
\rho(\halpha_2)=\alpha_1\quad
\rho(\halpha_4)=\alpha_2\quad
\rho(\halpha_3)=\rho(\halpha_5)=\alpha_3\quad
\rho(\halpha_1)=\rho(\halpha_6)=\alpha_4.
$$
It is the quotient by 
$$
\begin{array}{c}
\epsilon_2+\epsilon_3-\epsilon_1-\epsilon_4\qquad
\frac 1 2 \bigg
(\epsilon_1+\tilde\epsilon_6 +\epsilon_4-(\epsilon_2+\epsilon_3+3\epsilon_5)\bigg)
\end{array}
$$
Moreover
$$
\sharp \hPhi^+=36\quad
\sharp \widehat W = 51\,840\quad \sharp\Delta^2=2
\quad
\sharp W= 1\,152\quad \sharp\Phi^+=24\quad
\sharp\hW/W=45.
$$
We have 12 short positive roots in $F_4$:
$$
\epsilon_i \qquad \frac 1 2 (\epsilon_1\pm \epsilon_2\pm \epsilon_3\pm \epsilon_4)
$$

A working $\hy_0$ is
$$
\hy_0=\hs_1\, \hs_5\, \hs_3\, \hs_4\, \hs_2\, \hs_3\, \hs_4\, \hs_5\, \hs_4\, \hs_3\, \hs_1\, \hs_6.
$$
Indeed, one can check that $\ell(\hy_0)=12$ and $\dim(G\cap
\hy_0\hB\hy_0\inv)=16$.

Moreover, the only simple roots $\widehat\alpha$ such that $\dim(G\cap
\hw \hB\hw\inv)=17$ where $\hw=\hy_0s_\halpha$ are $\halpha_1$ and
$\halpha_6$. 
The so obtained divisors are distinct since $\sharp\Delta^2=2$. Set
$$
\Dc_0=\{\widehat\alpha_1,\,\widehat\alpha_6\}.
$$
One checks that
$$
\begin{array}{lcl}
  -\hy_0\halpha_1=\halpha_1+\halpha_2+2\halpha_3+2\halpha_4+\halpha_5
&\stackrel{\rho}{\longmapsto}&
\alpha_1+\alpha_2+3\alpha_3+\alpha_4;\\
-\hy_0\halpha_6=\halpha_1+\halpha_2+\halpha_3+2\halpha_4+2\halpha_5+\halpha_6
&\stackrel{\rho}{\longmapsto}&
\alpha_1+2\alpha_2+3\alpha_3+2\alpha_4.
\end{array}
$$

\bibliographystyle{alpha}

\bibliography{multspace}

\begin{center}
  -\hspace{1em}$\diamondsuit$\hspace{1em}-
\end{center}

\end{document}